\documentclass[reqno,oneside,11pt,hidelinks]{amsart}
\usepackage{amssymb,amsmath,amsthm,amsxtra,calc,bm, color}
\usepackage{enumerate}
\usepackage[margin=.95in]{geometry}

\usepackage[scr=boondoxo]{mathalfa}
\usepackage{mathtools}

\usepackage[labelsep=period]{caption}
\usepackage{mleftright}
\mleftright

\usepackage{enumitem}

\usepackage{nccmath}
\usepackage{cases}
\usepackage{calc}
\usepackage{graphicx}
\usepackage{hyperref, commath}
\usepackage[final]{microtype}

\renewcommand{\(}{\left(}
\renewcommand{\)}{\right)}
\newcommand{\spmod}[1]{\ensuremath{\,(#1)}}

\renewcommand{\|}{\big |}

\def\Z{\mathbb{Z}}

\def\R{\mathbb{R}}

\def\C{\mathbb{C}}

\def\SL{{\rm SL}}

\def\GL{{\rm GL}}

\newcommand{\pmfrac}[2]{\left(\mfrac{#1}{#2}\right)}

\newcommand{\pMatrix}[4]{\left(\begin{matrix}#1 & #2 \\ #3 & #4\end{matrix}\right)}

\renewcommand{\pmatrix}[4]{\left(\begin{smallmatrix}#1 & #2 \\ #3 & #4\end{smallmatrix}\right)}

\usepackage{xspace}

\newcommand{\Sh}{\mathcal{S}}

\renewcommand{\hat}{\widehat}
\renewcommand{\tilde}{\widetilde}

\DeclareMathOperator{\new}{new}

\def\ep{\varepsilon}

\newtheorem{theorem}{Theorem}[subsection]
\newtheorem{lemma}[theorem]{Lemma}

\newtheorem{proposition}[theorem]{Proposition}

\theoremstyle{remark}
\newtheorem*{remark}{Remark}

\newtheorem{definition}{Definition}
\numberwithin{equation}{section}

\mathtoolsset{showonlyrefs}

\usepackage{listings}

\lstdefinestyle{myListingStyle} 
    {
        basicstyle = \small\ttfamily,
        breaklines = true,
    }
\usepackage[nohyperlinks]{acronym}
\renewcommand{\url}[1]{#1}

\makeatletter
\newcommand*{\rom}[1]{\expandafter\@slowromancap\romannumeral #1@}
\makeatother

\makeatletter
\def\imod#1{\allowbreak\mkern5mu({\operator@font mod}\,\,#1)}
\makeatother
\allowdisplaybreaks


\makeatletter
\@namedef{subjclassname@2020}{%
  \textup{2020} Mathematics Subject Classification}
\makeatother

\begin{document}
\title{Explicit images for the Shimura Correspondence}

\author{Matthew Boylan \and Swati}
\address{Department of Mathematics, University of South Carolina, Columbia, SC, 29208, USA}
\email{boylan@math.sc.edu, s10@email.sc.edu}


\date{}

\begin{abstract}
For $(r, 6) = 1$ with $1 \leq r \leq 23$, and a non-negative integer $s$, we define
\[
\mathcal{A}_{r,s, N, \chi} = \{ \eta(z)^{r} f(z) : f(z) \in M_{s}(N, \chi)\}.
\]
In 2014, Yang showed that for $F \in \mathcal{A}_{r, s, 1, 1_N}$, we have $\textup{Sh}_{r}(F \mid V_{24}) = G \otimes \chi_{12}$ where $G\in S^{new}_{r+2s - 1}(\Gamma_{0}(6), - \left( \frac{8}{r}  \right), - \left( \frac{12}{r}  \right))$, where $\textup{Sh}_{r}$ is the $r$-th Shimura lift associated to the theta-multiplier. He proved a similar result for $(r,6) = 3$.\:His proofs rely on trace computations in integral and half-integral weights.

In this paper, we provide a constructive proof of Yang's result.  We obtain explicit formulas for $\mathcal{S}_{r}(F)$, the $r$-th Shimura lift associated to the eta-multiplier defined by Ahlgren, Andersen, and Dicks, when $1\leq r\leq 23$ is odd and $N = 1$.  We also obtain formulas for lifts of Hecke eigenforms multiplied by theta-function eta-quotients and lifts of Rankin-Cohen brackets of Hecke eigenforms with theta-function eta-quotients.
\end{abstract}

\maketitle

\section{Introduction and Statement of Results}
\subsection{Introduction}
\noindent
Let $N \geq 1$ and let $k \in \Z$ or $k \in \Z + \frac{1}{2}$. Let $\chi$ be the Dirichlet character modulo $N$ and let  $M_{k}(N, \chi)$ denote the space of holomorphic modular forms of weight $k$, on $\Gamma_{0}(N)$ and nebentypus $\chi$. Let $E_{k}(N, \chi)$ denote the subspace of Eisenstein series and $S_{k}(N, \chi)$ denote the subspace of cusp forms.

In 1999, Guo and Ono \cite{guo-ono} discovered a striking connection between squares of partition values and central critical values of quadratic twists of modular $L$-functions, modulo a prime $p$ where $13 \leq p \leq 31$. This serves as an important application of Shimura Correspondence: a family of maps between half-integer weight modular forms and integer weight forms. Furthermore, the works of Waldspurger \cite{waldspurger1, waldspurger2} and Kohnen and Zagier \cite{kohnen-zagier} establish a connection between squares of Fourier coefficients of a half-integer weight modular form $f$ and central critical values of quadratic twists of $L$-functions of the corresponding integer weight form $\textup{Sh}_{t}(f)$.

\medskip

For a prime $p$ such that $13 \leq p \leq 31$, let $G_{p}(z)$ denote the unique newform in the space ${S}_{p - 3}(6)$ whose Fourier expansion at infinity is as follows:
\[
G_{p}(z) := \sum_{n = 1}^{\infty} a_{p}(n)q^{n} := q + \left(\frac{2}{p} \right) \cdot 2^{\frac{p - 5}{2}} q^{2} +  \left(\frac{3}{p} \right) \cdot 3^{\frac{p - 5}{2}} q^{3} + \cdots
\]
Let $1 \leq r \leq 23$ with $(r, 6) = 1$, and a non-negative even integer $s$, define the following subspace:
\begin{equation}\label{FIA}
\mathcal{A}_{r, s, N, \chi} := \{\eta^{r}(z) f(z) : f(z) \in {M}_{s}(N, \chi)\} \subseteq  {S}_{r/2 + s}\left(N, \chi \nu_{\eta}^{r} \right).
\end{equation}
We note that 
\begin{align}
    \mathcal{A}_{r, s, 1, 1_N} = S_{r/2 + s}(1, \nu_{\eta}^r),  \: \mathcal{A}_{3r, s, 1, 1_N} = S_{3r/2 + s}(1, \nu_{\eta}^{3r}), \:\text{and}\\
    \mathcal{A}_{r, s, p, \left( \frac{\cdot}{p} \right)} = S_{r/2 + s}\left(p, \left( \frac{\cdot}{p} \right) \nu_{\eta}^r\right) \: \text{for all primes} \: p \geq 5.
\end{align}
Note that the works of Guo and Ono suggest that for $F \in \mathcal{A}_{r,s, 1, 1_N}$ and for all square-free $t \geq 1$, we have
\[
\textup{Sh}_{t}\left( F \mid V_{24} \right) \in {S}_{r + 2s - 1}(6) \otimes \left(  \frac{12}{\cdot} \right).
\]

Further, in 2010, Garvan \cite{garvan} showed for $1 \leq r \leq 23$ with $(r, 6) = 1$, non-negative even integers $s$, and primes $p \geq 5$, 
\[
F \mid V_{24}  \mid T_{p^2} \in \mathcal{A}_{r,s, 1, 1_N}.
\]
It follows that if $\textup{dim} \:{M}_{s}(1) = 1$, and $0 \neq f(z) \in {M}_{s}(1)$, then the function $F \mid V_{24}$ is a Hecke eigenform for ${S}_{r/2 + s}(576, \left(  \frac{12}{\cdot} \right) \nu_{\theta}^{r + 2s})$.
The forms $G_{p}(z)$ for all $13 \leq p \leq 31$ in the Guo-Ono paper are eigenforms for all $T_{p^{2}}$ for this reason. 

Moreover, in 2010, J.J. Webb \cite{webb} extended the results of Guo and Ono to account for the Ramanujan Congruences.\:In other words, he established similar connections between the partition function and ratios of central critical values of modular L-functions for primes $5, 7$ and $11$.

The works of Guo and Ono motivated subsequent work on Shimura images on the subspaces $\mathcal{A}_{r,s, 1, 1_N}$. In 2014, Yang proved the following isomorphism.
\begin{theorem}[Theorems 1 and 2, \cite{yang1}] \label{Yang}
Let $r$ be an integer satisfying $1 \leq r \leq 23$ , and let $s$ be a non-negative even integer.
\begin{enumerate}
    \item Let $(r,6) = 1$.  Let $F \in \mathcal{A}_{r,s, 1, 1_N}$.
\[
\textup{Sh}_{r}\left( F \mid V_{24} \right) \in {S}_{r + 2s - 1}^{new} \left(6, - \left(  \frac{8}{r} \right), - \left( \frac{12}{r} \right) \right) \otimes \left( \frac{12}{\cdot} \right).
\]
    \item Let $r$ be an odd integer with $0 < r < 8$.  Let $F \in \mathcal{A}_{3r,s, 1, 1_N}$.
\[
\textup{Sh}_{r}\left( F \mid V_{8} \right) \in {S}_{3r + 2s - 1}^{new} \left(2, - \left(  \frac{8}{r} \right)\right) \otimes \left( \frac{-4}{\cdot} \right).
\]
\item The maps (1) and (2) are isomorphisms.
\end{enumerate}
\end{theorem}
Here, ${S}_{r + 2s - 1}^{new} (6, \epsilon_{2}, \epsilon_{3})$ is the space of newforms of weight $r + 2s - 1$ on $\Gamma_{0}(6)$ that are eigenfunctions with eigenvalues $\epsilon_{2} = - \left(  \frac{8}{r} \right) $ and $\epsilon_{3} = - \left( \frac{12}{r} \right) $ for the Atkin-Lehner operator $W_{2}^{6}$ and $W_{3}^{6}$, respectively. 
\begin{remark}
Yang's work shows that the multiplicity-one property holds for the spaces $\mathcal{A}_{r,s, 1, 1_N}$ defined in Theorem \ref{Yang}. Multiplicity - one does not, in general, hold for spaces of half-integer weight modular forms. 
\end{remark}
Before we describe our work, we give some of the benchmark results regarding
explicit Shimura images of a certain class of half-integer modular forms.

\medskip

The story begins with Selberg who in an published manuscript obtained a formula for the Shimura lift $\textup{Sh}_{t}$, of an eigenform times a theta function, for $t = 1.$
\begin{theorem}\label{Selberg}
Suppose $\displaystyle{f(z) = \sum_{n = 1}^{\infty} a(n)q^{n}} \in {S}_{k}(1)$ is a normalized eigenform for all Hecke operators $T_{p}$. Define
\[
F(z) = f(4z) \theta(z) \in {S}_{k + 1/2} (4, \nu_{\theta}^{2k + 1}).
\]
Then we have $\textup{Sh}_{1}(F(z)) = f(z)^{2} - 2^{k - 1} f(2z)^{2} \in {S}_{2 k}(2)$.
\end{theorem}
In 1989, Cipra \cite{cipra} generalized Selberg's work by considering Shimura images for half-integer weight forms that are products of a newform times a theta function of arbitrary level and character of prime power modulus.

\begin{theorem}\label{Cipra}
Let $\displaystyle{f(z) = \sum_{n = 1}^{\infty} a(n)q^{n} \in {S}_{k}(N, \chi)}$ be a newform, where $\chi$ is a character modulo $N$ with $\chi(-1) = (-1)^{k}$. Let $\chi_{r}$ be a primitive character modulo $r$, where $r = p^{m}, m \geq 1$, and $p$ is prime. Define the generalized theta function, 
\begin{equation} \label{gen_theta}
\theta(\chi_{r};z) = \displaystyle{\sum_{n \in \mathbb{Z}}} \chi_{r}(n)n^{\nu}q^{n^{2}} \in M_{1/2 + \nu}\left( 4r^2, \chi_{r} \chi_{-4}^{\nu} \nu_{\theta}^{2 \nu + 1} \right),
\end{equation}
where $\chi_{r}$ is a Dirichlet character of modulus $r$, where $\nu = 0, 1$ is chosen such that $\chi_{r}(-1) = (-1)^{\nu}$. Take $\mu \geq m$, and define
\[
F(z) := f(4 p^{\mu}z) \theta(\chi_{r};z)
\]
Further, define
\[
g(z) := \sum_{n = 1}^{\infty} c(n)q^{n} = \begin{cases}
f(z) f(p^{\mu}z) & \text{if} \ \nu = 0 \\
\frac{1}{2 \pi i}(f'(z) f(p^{\mu}z) - p^{\mu}f(z) f'(p^{\mu}z)) & \text{if} \ \nu = 1 \\
\end{cases}
\]
and let $G(z) = g(z) \otimes \chi_{r}$. 
Let $N' = \textup{lcm} (Np^{\mu}, r^{2})$. Then, $F \in {S}_{k + \nu + 1/2}(4N', \chi \chi_{r} \chi_{-4}^{k + \nu} \nu_{\theta}^{2(k + \nu) + 1})$, and we have 
\[
\textup{Sh}_{1}(F(z)) = G(z) - 2^{k + \nu - 1} \chi_{r}(2) \chi(2) G(2z) \in {S}_{2(k + \nu)} (2 N'r, \chi^{2} \chi_{r}^{2}).
\]
\end{theorem}
In 2007, Hansen and Naqvi \cite{HN} generalized Cipra's work by considering Shimura lifts of theta functions of arbitrary Dirichlet character modulo any positive integer multiplied by Hecke eigenforms.

\begin{theorem}\label{HNN}
Let $\chi_{r}$ be a Dirichlet character modulo $r$, and write $\displaystyle{\chi_{r} = \chi_{r_{1}} \chi_{r_{2}} \cdots \chi_{r_{l}}}$, where $\chi_{r_{i}}$ is a Dirichlet character modulo $p_{i}^{\alpha_{i}}$, for $1 \leq i \leq l$. Let $f(z) \in {M}_{k}(N, \chi)$ be a normalized Hecke eigenform, and define
\[
F(z) := f(4rz)\theta(\chi_{r};z) \in {S}_{k + \nu + 1/2} \left( 4 N'r^{2}, \chi \chi_{r} \chi_{-4}^{k + \nu} \nu_{\theta}^{2(k + \nu) + 1} \right) \: \text{with} \: N' = \frac{N}{(N,r)}.
\]
Further, define $g(z)$ as follows:
\[
g(z) = \begin{cases}
    \displaystyle{\sum_{\substack{d \mid r \\ (d,r/d) = 1}}} \chi_{d}(-1) f(dz) f(rz/d), & \text{if} \ \nu = 0, \\
    \displaystyle{\frac{1}{\pi i} \sum_{\substack{d \mid r \\ (d,r/d) = 1}}} \chi_{d}(-1) d f'(dz) f(rz/d), & \text{if} \ \nu = 1.
\end{cases}
\]
Let $g_{\chi_{r}}(z)$ be the $\chi_{r}$-twist of $g(z)$. Then, we have the Shimura image as follows:
\begin{equation} \label{HH}
\textup{Sh}_{1}(F(z)) = g_{\chi_{r}}(z) - 2^{k + \nu - 1} \chi(2) \chi_{r}(2) g_{\chi_{r}}(2z) \in {S}_{2(k + \nu)} (2 N' r^{2}, \chi_{r}^{2} \chi^{2}).
\end{equation}
\end{theorem}

 In 2021, Pandey and B. Ramakrishnan \cite{PR} obtained a generalization of Hansen and Naqvi's work for the $t$-th Shimura lifts for any positive square-free integer $t$.
Their results are conditional upon the fact that the level of the half-integer weight forms (whose $t$-th Shimura image is being considered) should be divisible by $t$.

In a recent work, Xue \cite{huixue} proved a Selberg-type identity related to the Rankin-Cohen Bracket as defined in \eqref{defn_Rankin-Cohen} of $\textup{Sh}^{+}_{1}$ of a normalized Hecke eigenform and the theta function where  $\textup{Sh}^{+}_{1}$ denotes the Shimura lift of the Kohnen's plus-space. Furthermore, Wang \cite{wang} obtained a generalization of Xue's work for the $D$-th Shimura lift $\textup{Sh}^{+}_{D}$ of normalized Hecke eigenform and generalized theta functions $\theta(\chi_{r};Dz)$,  for a fundamental discriminant $D$ where $D \mid N$ and $(D, r) = 1$ . More generally, he showed that $t$-th Shimura lift $\Sh_t$ of a normalized Hecke eigenform and $\theta(\chi_{r};tz)$, for a square-free integer $t$ with $t \mid N$ and $(t, r) = 1$.

\subsection{Notation and Statement of Results}
We now turn to our work. We require some notation. Let $\chi$ be a Dirichlet character modulo $N \geq 1$ and let $r \in \mathbb{Z}$. We note that $k \in \Z$ if and only if $r$ is even and $k \in \Z + \frac{1}{2}$ if and only if $r$ is odd. For $k \geq 0$, let $M_{k + 1/2}(N, \chi \nu_{\eta}^{r})$ denote the space of modular forms of weight $k + 1/2$ on $\Gamma_{0}(N)$, transforming with respect to the multiplier $\chi \nu_{\eta}^{r}$. These spaces are trivial unless $\chi(-1) = \left( \frac{-1}{r} \right) (-1)^{k}$. Equivalently, $M_{k + 1/2}(N, \chi \nu_{\eta}^{r}) = \{0\}$ unless $r$ is odd. 

\medskip
Let $(N, 6) = 1$. We define 
\begin{equation}\label{eq:epdef}
\ep_{2, r, \chi} =-\chi(2)\pmfrac{8}{r/(r,3)}, \ \ \ \ 
\ep_{3, r, \chi} =-\chi(3)\pmfrac{12}{r}.
\end{equation}
We let $S_{2k}^{\new 2, 3}(6N, \chi^2, \epsilon_{2, 1, \chi}, \epsilon_{3, 1, \chi})$ denote the space of cusp forms of weight $2k$ on $\Gamma_{0}(6N)$, that are new at primes 2 and 3 as defined in Lemma \ref{new_criterion}, and eigenfunctions for $W_{2}^{6N}$ and $W_{3}^{6N}$  with eigenvalues $\epsilon_{2, r, \chi}, \epsilon_{3, r, \chi}$, resp.

\medskip

Let $f$ and $g$ be two smooth functions on $\C$. For given real numbers $k_{1}, k_{2}$, and a non-negative integer $w$, we define the $w$-th Rankin-Cohen Bracket \cite{zagier} of $f$ and $g$ by  
\begin{equation} \label{defn_Rankin-Cohen}
[f, g]_{w}(z) = \sum_{r = 0}^{w} (-1)^{r} {w + k_1 - 1 \choose w - r}{w + k_2 - 1 \choose r} f^{(r)}(z) g^{(w - r)}(z),
\end{equation}
where $f^{(r)}(z) = \frac{1}{(2 \pi i)^r} \frac{d^r(f(z))}{dz^{r}} = (\Theta^{r}f)(z)$ denotes the normalized $r$-th derivative of $f$ with respect to $z$. In particular,
\begin{equation}\label{rankin_01}
[f, g]_0 = f g, \quad [f,g]_1 = k \: \left(f(z) \Theta(g(z)) - \Theta(f(z)) g(z)    \right)    
\end{equation}

We observe that Rankin-Cohen brackets defined in \cite{Cohen} are connected to \eqref{defn_Rankin-Cohen} via $F_{w}^{(k_1, k_2)} = (-1)^{w} w! [f(z), g(z)]_{w}$. Further, if we let $f$ and $g$ are two modular forms of weights $k_{1}, k_{2} \in \Z$ or $\Z + \frac{1}{2}$ resp. on any congruence subgroup $\Gamma$ of $\SL_{2}(\Z)$, then by Theorem 7.1 of \cite{Cohen}, we have $[f, g]_{w} \in M_{k_1 + k_2 + 2w}(\Gamma)$. Moreover, if $w \geq 1$, $[f, g]_{w}$ is a cusp form.

\medskip

Let $r$ be an odd integer with $1 \leq r \leq 23$. Let $\mathcal{A}_{r, s, N, \chi}$ be as defined above. Our next result provides a list of explicit formulas for the first Shimura image $\mathcal{S}_{1}$ as defined in Theorem \ref{AAD1} and Theorem \ref{AAD2} corresponding to eta-quotients that are theta functions associated with the eta-multiplier on spaces $\mathcal{A}_{1,k, N, \chi}$ and $\mathcal{A}_{3,k, N, \chi}$, respectively.

\begin{theorem} \label{first_Shimura_image}
 Let $g(z) \in M_{k}(N, \chi)$ be a normalized Hecke eigenform for all $T_p$. Let $\tilde{g}(z) = (g \mid U_{2} \mid V_{2})(z)$.

 \begin{enumerate}
 \item The Shimura image $\Sh_{1}$ for the product of weight 1/2 eta-quotients with the eta-multiplier as specified in Table \ref{tab:eta_even} and a normalized Hecke eigenform is given by \label{first_Shimura_image:1}:
 \vspace{ 2mm}
\begin{enumerate}[label=(\alph*)]
 \item  $\mathcal{S}_{1}(\eta(z) g(z)) = g(z) g(6z) - g(2z) g(3z).$ \label{first_Shimura_image:1a}
\item $\Sh_{1} \left( \frac{\eta(z) \eta(4z)}{\eta(2z)} \: g(z)  \right) = (2  g(z) \tilde{g}(2z) - g(z) g(2z)) \otimes \left( \frac{-8}{\cdot} \right).$ \label{first_Shimura_image:1b}
\item $\Sh_{1} \left( \frac{\eta^{2}(2z)}{\eta(z)} \: g(z)  \right) = (g(z) g(2z)) \otimes \left( \frac{-4}{\cdot} \right).$ \label{first_Shimura_image:1c}
\item $\Sh_{1} \left( \frac{\eta(2z) \eta^{2}(3z)}{\eta(z) \eta(6z)} \: g(z)  \right) = (g(z) g(6z) + g(2z) g(3z)) \otimes \left( \frac{12}{\cdot} \right).$ \label{first_Shimura_image:1d}
\item $\Sh_{1} \left( \frac{\eta^{3}(2z)}{\eta(z) \eta(4z)} \: g(z)  \right)$ = 
\[
\left( 2 g(z) \tilde{g}(6z) - g(z) g(6z) \right) \otimes \left( \frac{8}{\cdot} \right) - \left(g(2z) \cdot \left(g(3z) \otimes \left(\mfrac{-4}{\cdot}\right) \right) \right) \otimes \left( \frac{-8}{\cdot} \right).
\] \label{first_Shimura_image:1e}
\item $\Sh_{1} \left( \frac{\eta(z) \eta(4z) \eta^{5}(6z)}{\eta^{2}(2z) \eta^{2}(3z) \eta^{2}(12z)} \: g(z)  \right)$ = 
\[
\left( 2  g(z) \tilde{g}(6z) - g(z) g(6z) \right) \otimes \left( \mfrac{24}{\cdot} \right) + \left(g(2z) \left(g(3z) \otimes \left( \mfrac{-4}{\cdot} \right) \right) \right) \otimes \left( \frac{-24}{\cdot} \right).
\]
\label{first_Shimura_image:1f}
\end{enumerate}
\item The Shimura image $\Sh_{1}$ for the product of weight 3/2 eta-quotients with the eta-multiplier as specified in Table \ref{tab:eta_odd} and a normalized Hecke eigenform is given by \label{first_Shimura_image:2}:
\vspace{ 2mm}
\begin{enumerate}[label=(\alph*)]
   \item $\mathcal{S}_{1} (\eta^{3}(z) g(z)) = \mfrac{1}{k} \left([g(2z), g(z)]_{1}\right)$. \label{first_Shimura_image:2a}
\item $\Sh_{1} \left( \frac{\eta^9(2z)}{\eta^{3}(z) \eta^{3}(4z)} \: g(z)  \right) = 
\mfrac{1}{k} \left( 2 \: [\tilde{g}(2z), g(z)]_{1}  - [g(2z), g(z)]_{1} \right) \otimes \left( \frac{8}{\cdot} \right)$. \label{first_Shimura_image:2b}
\item $\Sh_{1} \left( \frac{\eta^{5}(z)} {\eta^{2}(2z)} \: g(z)  \right) = \mfrac{1}{k} \left([g(6z), g(z)]_{1}  + [g(3z), g(2z)]_{1} \right) \otimes \left( \frac{-4}{\cdot} \right)$. \label{first_Shimura_image:2c}
\item $\Sh_{1} \left( \frac{\eta^{13}(2z)}{\eta^{5}(z) \eta^{5}(4z)} \: g(z)  \right) = $
\[
\mfrac{1}{k} \left( 2 \: [\tilde{g}(6z), g(z)]_{1}  - [g(6z), g(z)]_{1} \right) \otimes \left( \frac{-8}{\cdot} \right) + \mfrac{1}{k} \left( \left[ g(3z) \otimes \left( \mfrac{-4}{\cdot} \right), g(2z)  \right]_{1} \right) \otimes \left( \frac{8}{\cdot} \right).
\] \label{first_Shimura_image:2d}
\end{enumerate}
\end{enumerate}
\end{theorem}
Here, $\Theta$ refers to the Theta operator, as defined in \eqref{defn_theta}. We recall that $\Theta(f(z)) = \frac{1}{2 \pi i} \frac{df(z)}{dz}$.

\medskip

\begin{theorem}\label{first_image}
Let $g \in M_k(N, \chi)$ be a normalized Hecke eigenform for all $T_p$.
 \begin{enumerate}
    \item  Let $(N, 6) = 1$. Then we have
    \[
    \mathcal{S}_{1}(\eta(z) g(z)) = g(z) \: g(6z) - g(2z) \: g(3z) \in S_{2k}^{\new 2, 3} \: \left(6N, \chi^2, \epsilon_{2, 1, \chi} , \epsilon_{3, 1, \chi} \right).
    \]
    \item Let $(N, 2) = 1$, Then we have
    \[
    \mathcal{S}_{1} (\eta^{3}(z) g(z)) = g(2z) \: \Theta(g(z)) - g(z) \: \Theta(g(2z)) \in S_{2k + 2}^{\new 2} \: \left(2N, \chi^2, \epsilon_{2, 3, \chi} \right).
    \]
\end{enumerate}
\end{theorem}
{Remarks:}
We note that for $(N, 6) = 3$, we have $S_1(\eta(z) g(z)) \in S_{2k}^{\new 2}(6N, \chi^2, \epsilon_{2, 1, \chi})$ and for $(N, 6) = 2$, we have $S_1(\eta(z) g(z)) \in S_{2k}^{\new 3}(6N, \chi^2, \epsilon_{3, 1, \chi})$.

\medskip

Our next result focuses on explicit formulas for the $r$-th Shimura image on the subspaces $\mathcal{A}_{r,s, 1, 1_N}$. Let $1 \leq r \leq 23$ with $(r, 6) = 1$. Let $s \geq 0$ be even and $k = (r - 1)/2 + s$. 
We observe that $M_{k}\left(r, \left( \frac{\cdot}{r} \right)\right)$ has a basis $\{g_{1}, \ldots, g_{d}\}$ of normalized Hecke eigenforms with 
    \begin{align*}
    g_{1} = \mfrac{L\left(1 - k, \left( \mfrac{\cdot}{r} \right)\right)}{2} + \sum_{n \geq 1} \sum_{d \mid n} \left( \frac{d}{r} \right) d^{k - 1} q^{n} \:\: \: \: , \:\: \: \:
    g_{2} = \sum_{n \geq 1} \sum_{d \mid n} \left( \frac{n/d}{r} \right) d^{k - 1} q^{n},
    \end{align*}
    and newforms $g_{i}$ for $3 \leq i \leq d$.
    
Let $f(z) \in M_{s}(1)$. Then, there exists $\alpha_{i}$ for $1 \leq i \leq d$ such that
    \begin{equation} 
    \left( \frac{\eta^{r}(z)}{\eta(rz)} f(z) \right) \mathrel{\bigg|_{k}} U_{r} = \sum_{i = 1}^{d} \alpha_i g_{i}(z).
    \end{equation}
For all $1 \leq i \leq d$, let 
$G_{i}(z) = g_{i}(z) g_{i}(6z) - g_{i}(2z) g_{i}(3z) \in S_{2k}(6r).$
\begin{theorem} \label{image_(r,6) = 1} Assume the notation above. 
    \begin{enumerate}
        \item For all $1 \leq i \leq d$, we have $G_{i}(z) \in S_{2k}^{\new 2, 3}\left(6r, -\left( \frac{8}{r} \right), - \left( \frac{12}{r} \right) \right)$.
       
        \item  $\mathcal{S}_{r}(\eta^r(z) f(z)) = \displaystyle{\sum_{i = 1}^{d}} \alpha_{i} G_{i} \in S_{2k}^{\new}\left(6, -\left( \frac{8}{r} \right), - \left( \frac{12}{r} \right) \right).$ 
      
 \end{enumerate}
\end{theorem}
\noindent
Let $r$ be an odd integer with $0 < r < 8$. Let $k = 3(r - 1)/2 + s$. For all $1 \leq i \leq d$, we let 
    \[
     G_{i}(z) = g_{i}(2z) \: \Theta (g_{i}(z)) - g_{i}(z) \: \Theta (g_{i}(2z)) \in S_{2k + 2}(2r).
    \]
\begin{theorem} \label{image_(3|r)} Assume the notation above. 
\begin{enumerate}
        \item   For all $1 \leq i \leq d$, we have $G_{i}(z) \in S_{2k + 2}^{\new  2}\left(2r, - \left( \frac{8}{r}   \right) \right)$.
        
   \item $\mathcal{S}_{r}(\eta^{3r}(z) f(z)) = \displaystyle{\sum_{i = 1}^{d}} \alpha_{i} G_{i} \in  S_{2k + 2}^{\new} \left(2, - \left( \frac{8}{r} \right) \right)$.
\end{enumerate}
\end{theorem}
{\bf{Remarks:}}
\begin{enumerate}
        \item Using Proposition 3.2 of \cite{martin}, one can verify that for $(N, 6) = 1$, we have 
        \[
        \dim \: S_{k + 1/2}(N, \chi \nu_{\eta}^r) = \dim \: S_{2k}^{\new} \left(6, \chi^2, -\left( \frac{8}{r} \right), - \left( \frac{12}{r} \right) \right),
        \]
        and for $N$ odd, we have 
        \[
        \dim \: S_{k + 3/2}(N, \chi \nu_{\eta}^{3r}) = \dim \: S_{2k + 2}^{\new} \left(2, \chi^2, -\left( \frac{8}{r} \right) \right).
        \]
\item Let $\chi$ be a real and primitive character and $(r, N) = 1$. Let $f(z) \in M_{s}(N, \chi)$.
\begin{enumerate}
\item Theorem \ref{image_(r,6) = 1} and Theorem \ref{image_(3|r)} hold true.
\item Using Lemma 6 of \cite{winnie-li}, we deduce that for $k = (r - 1)/2 + s$ and $(N, 6) = 1$,
\[
S_{2k}^{\new} \left(6N, \epsilon_{2, r, \chi}, \epsilon_{3, r, \chi} \right) \subseteq S_{2k}^{\new} \left(6rN, \epsilon_{2, r, \chi}, \epsilon_{3, r, \chi} \right),
\]
and for for $k = 3(r - 1)/2 + s$ and $N$ odd, 
\[
S_{2k + 2}^{\new} \left(2N, \epsilon_{2, r, \chi} \right) \subseteq S_{2k + 2}^{\new} \left(2rN, \epsilon_{2, r, \chi} \right).
\]
\end{enumerate}
\end{enumerate}
\medskip

The following result is a generalized version of Theorem \ref{first_Shimura_image} in the Rankin-Cohen bracket setting.
\begin{theorem} \label{rankin_bracket}
Let $g(z) \in M_{k}(N, \chi)$ be a normalized Hecke eigenform for all $T_p$. Let $w \geq 0$ be an integer. We define
\[
\alpha_{k, w} = {k + w - 1 \choose w} {k + 2w - 1 \choose 2w}^{-1} \quad \text{and} \quad \beta_{k, w} = {k + w - 1 \choose w} {k + 2w \choose 2w + 1}^{-1}.
\]
Then we have
\begin{enumerate} 
\item The Shimura image $\Sh_{1}$ for the Rankin-Cohen bracket of a normalized Hecke eigenform and a weight 1/2 eta-quotient with the eta-multiplier as specified in Table \ref{tab:eta_even} is given by
\medskip
\begin{enumerate} [label=(\alph*)]
\item $\mathcal{S}_{1}\left( [g(z), \eta(z)]_{w} \right) = 24^{-w} \: \alpha_{k, w} \: \left( [g(6z), g(z)]_{2w} - [g(3z), g(2z)]_{2w}  \right)$. \\
\item $\Sh_{1} \left( \left[g(z), \frac{\eta(z) \eta(4z)}{\eta(2z)}  \right]_{w} \right) = 8^{-w} \: \alpha_{k, w} \: \left(2 \: [\tilde{g}(2z), g(z)]_{2w} - [g(2z), g(z)]_{2w}\right) \otimes \left( \frac{-8}{\cdot} \right).$ 
\item $\Sh_{1} \left( \left[g(z), \frac{\eta(2z)^{2}}{\eta(z)}\right]_{w} \right) = 8^{-w} \: \alpha_{k, w} \: \left([g(2z), g(z)]_{2w} \right)\otimes \left( \frac{-4}{\cdot} \right)$
\item $\Sh_{1} \left( \left[ g(z), \frac{\eta(2z) \eta(3z)^{2}}{\eta(z) \eta(6z)} \right]_{w} \right) = 24^{-w} \: \alpha_{k, w} \: ([g(6z), g(z)]_{2w} + [g(3z), g(2z)]_{2w}) \otimes \left( \frac{12}{\cdot} \right).$
\item $\Sh_{1} \left( \left[ \: g(z), \frac{\eta^{3}(2z)}{\eta(z) \eta(4z)} \right]_{w} \right)$ = 
\[
 24^{-w} \alpha_{k, w} \left( 2 \: [\tilde{g}(6z), g(z)]_{2w} - [g(6z), g(z)]_{2w} \right) \otimes \left( \mfrac{8}{\cdot} \right) - \left(\left[g(3z) \otimes \left( \mfrac{-4}{\cdot} \right), g(2z) \right]_{2w} \right) \otimes \left( \frac{-8}{\cdot} \right).
\]
\item $\Sh_{1} \left( \left[ \: g(z), \frac{\eta(z) \eta(4z) \eta^{5}(6z)}{\eta^{2}(2z) \eta^{2}(3z) \eta^{2}(12z)} \right]_{w} \right)$ = 
\[
 24^{-w} \: \alpha_{k, w} \: \left( 2 \: [\tilde{g}(6z), g(z)]_{2w} - [g(6z), g(z)]_{2w} \right) \otimes \left( \mfrac{24}{\cdot} \right) + \left(\left[g(3z) \otimes \left( \mfrac{-4}{\cdot} \right), g(2z) \right]_{2w} \right) \otimes \left( \frac{-24}{\cdot} \right).
\]
\end{enumerate}
\item The Shimura image $\Sh_{1}$ for the Rankin-Cohen bracket of a normalized Hecke eigenform and a weight 3/2 eta-quotient with the eta-multiplier as specified in Table \ref{tab:eta_odd} is given by
\medskip
\begin{enumerate} [label=(\alph*)]
\item $\mathcal{S}_{1}\left( [g(z), \eta(z)^{3}]_{w} \right) = 8^{-w} \: \beta_{k, w} \: \left( [g(2z), g(z)]_{2w + 1} \right)$.
\item $\Sh_{1} \left( \left[g(z), \frac{\eta(2z)^9}{\eta(z)^{3} \eta(4z)^{3}} \right]_{w} \right) = 
8^{-w} \: \beta_{k, w} \: \left( 2 \: [\tilde{g}(2z), g(z)]_{2w + 1}  - [g(2z), g(z)]_{2w + 1} \right) \otimes \left( \frac{8}{\cdot} \right)$.
\item $\Sh_{1} \left( \left[ g(z), \frac{\eta(z)^{5}} {\eta(2z)^{2}}\right]_{w} \right) = 24^{-w} \: \beta_{k, w} \: \left( [g(6z), g(z)]_{2w + 1}  + [g(3z), g(2z)]_{2w + 1} \right) \otimes \left( \frac{-4}{\cdot} \right)$. 
\item $\Sh_{1} \left( \left[ g(z), \frac{\eta(2z)^{13}}{\eta(z)^{5} \eta(4z)^{5}} \right]_{w} \right) = $
\[
24^{-w} \: \beta_{k, w} \: \left( 2 \: [\tilde{g}(6z), g(z)]_{2w + 1}  - [g(6z), g(z)]_{2w + 1} \right) \otimes \left( \frac{-8}{\cdot} \right) + \left( \left[ g(3z) \otimes \left( \mfrac{-4}{\cdot} \right), g(2z)  \right]_{2w + 1} \right) \otimes \left( \frac{8}{\cdot} \right).
\]
\end{enumerate}
\end{enumerate}
\end{theorem}
We now outline the plan for the rest of the paper. In Section 2, we provide explicit numerical examples to illustrate Theorem \ref{image_(r,6) = 1}, and Theorem \ref{image_(3|r)}. In Section 3, we provide necessary background on modular forms and Shimura lifts associated to the theta multiplier and the eta-multiplier. In Section 4, we prove our main theorems. 

\medskip

\begin{center}
    \footnotesize{ACKNOWLEDGEMENTS}
\end{center}
\par
We thank Scott Ahlgren for his helpful comments and discussions.

\section{Examples}
\medskip
{\bf{Example 1:}}
We compute $\mathcal{S}_{5}(\eta^{5}(z) E_{4}(z))$. Here, $r = 5$, $f = E_{4}$, and $s = 4$.  We require the following forms in $M_6 \left(5 , \left(\frac{\cdot}{5} \right)\right)$. 
\[
f_0(z) = \frac{\eta(z)^{15}}{\eta(5z)^3}, \ \ 
f_1(z) = \eta(z)^9\eta(5z)^3, \ \ 
f_2(z) = \eta(z)^3\eta(5z)^9, \ \ 
f_3(z) = \frac{\eta(5z)^{15}}{\eta(z)^3}.  
\]
We observe that 
\begin{align*}
&g_1(z) = -\frac{67}{5}f_0(z) - 5^2\cdot 8 f_1(z) - 5^4f_2(z), \\
&g_2(z) = f_1(z) + (9 - 2\sqrt{-11})f_2(z), \\
&g_3(z) = f_1(z) + (9 + 2\sqrt{-11})f_2(z), \\
&g_4(z) = f_1(z) + 5\cdot 8 f_2(z) + 5\cdot 67 f_3(z)
\end{align*}
are normalized Hecke eigenforms in this space for all $T_p$. We further require
\[
\alpha = -\frac{3 \cdot 5^4}{67}\left(103 + \frac{3 \cdot 283 \sqrt{-11}}{11}\right).  
\]
\noindent
We compute
\[
\left(\frac{\eta(z)^5}{\eta(5z)}E_4(z)\right) \mathrel{\bigg|_{6}} U_5
= -\frac{5}{67}g_1(z) + \overline{\alpha} g_2(z) + \alpha g_3(z).  
\]
For each $i$, $1 \leq i \leq 4$, we set 
\[
G_i(z) = g_i(z)g_i(6z) - g_i(2z)g_i(3z) \in M_{12}(\Gamma_0(30)).  
\]
Then the theorem gives
\begin{align*}
&\mathcal{S}_5(\eta(z)^5 E_4(z)) = \mathcal{S}_{1} ((\eta(z)^{5} E_{4}(z)) \mid_{13/2} U_{5}) 
= \mathcal{S}_{1} \left( \eta(z) \left( -\frac{5}{67}g_1(z) + \overline{\alpha} g_2(z) + \alpha g_3(z) \right)\right),  \\
&= \frac{-5}{67} G_{1}(z) + \overline{\alpha} G_{2}(z) + \alpha G_{3}(z) = q - 32 q^{2} - 243 q^{3} + 1024 q^{4} + 5766 q^{5} + \mathcal{O}(q^6).
\end{align*}
Lastly,
$\text{Tr}_{6}^{30}(\mathcal{S}_5(\eta(z)^5 E_4(z))) = [\Gamma_0(6):\Gamma_0(30)] G(z) = 6G(z) \in M_{12}(6)$. Using LMFDB, we find that $\mathcal{S}_5(\eta(z)^5 E_4(z))$ coincides with the newform 6.12.a.a. 

\medskip

{\bf{Example 2:}} We compute $\mathcal{S}_{3}(\eta^{9}(z) E_{6}(z))$. Here, $r = 9$, $f = E_{6}$, and $s = 6$.  We require the following forms in $M_9 \left(3 , \left(\frac{-3}{\cdot} \right)\right)$. 
\[
f_0(z) = \frac{\eta(z)^{27}}{\eta(3z)^9}, \ \ 
f_1(z) = \eta(z)^{3} \eta(3z)^{15}, \ \ 
f_2(z) = \eta(z)^{15} \eta(3z)^{3}, \ \ 
f_3(z) = \frac{\eta(3z)^{27}}{\eta(z)^9}.  
\]
We observe that 
\begin{align*}
&g_1(z) = 270 \:f_1(z) + f_2(z) + 7281 \: f_3(z), \\
&g_2(z) = (809/27) \: f_0(z) + 2187 \: f_1(z) + 810 \: f_{2}(z), \\
&g_3(z) = (15 - 6 \sqrt{-14}) \: f_1(z) + f_2(z), \\
&g_4(z) = (15 + 6 \sqrt{-14}) \: f_1(z) + f_2(z)
\end{align*}
are normalized Hecke eigenforms in this space for all $T_p$.

Let $\alpha_{2} = \frac{27}{809}, \: \alpha_{3}$ = $\frac{5296914}{809} - \frac{6348861 \sqrt{-14}}{809}$ and $\alpha_{4}$ = $\overline{\alpha_{3}}$.
We compute
\[
\left(\frac{\eta(z)^9}{\eta^{3}(3z)}E_6(z)\right) \mathrel{\bigg|_{9}} U_3
= \alpha_{2} g_2(z) + \alpha_{3} g_3(z) + \alpha_{4} g_4(z).  
\]
For each $i$, $1 \leq i \leq 4$, we set 
\[
G_i(z) = g_i(2z) \: \Theta(g_i(z)) - g_i(z) \: \Theta(g_i(2z)) \in M_{20}(6).  
\]
Then the theorem gives
\begin{align*}
&\mathcal{S}_3(\eta^{9}(z) E_6(z)) = \mathcal{S}_{1} ((\eta(z)^{9} E_{6}(z)) \mid_{21/2} U_{3}) 
= \mathcal{S}_{1} \left( \eta^{3}(z) \left( \alpha_{2} g_2(z) + \alpha_{3} g_3(z) + \alpha_{4} g_4(z) \right)\right),  \\
&= \alpha_{2} G_{2}(z) + \alpha_{3} G_{3}(z) + \alpha_{4} G_{4}(z) = q - 512 q^{2} - 13092 q^{3} + 262144 q^4 + 6546750 q^5 + \mathcal{O}(q^6).
\end{align*}
$\text{Tr}_{2}^{6}(\mathcal{S}_{3}(\eta(z)^9 E_6(z))) = [\Gamma_0(2) : \Gamma_0(6)] G(z) = 4 G(z) \in M_{20}(2)$. Using LMFDB, we find that $\mathcal{S}_3(\eta(z)^9 E_6(z))$ coincides with the newform 2.20.a.a. 
\section{Background} \label{background}
For background on classical modular forms, one may consult \cite{Cohen2017ModularFA}. For details on modular forms with theta-multiplier and eta-multiplier, we refer the reader to \S 3 of \cite{AAD} and \S 2 of \cite{yang2}, repectively.
\subsection{Operators on Modular Forms}
We define some important operators acting on spaces of modular forms that play a central role in our work.  Throughout this section, we let $k$ to be an integer unless otherwise specified. Let $N\geq 1$ be an integer, and we let $\chi$ be a Dirichlet character modulo $N$.  For  functions $f$ on the upper half-plane $\mathbb{H}$ and for all $\gamma = \pMatrix{a}{b}{c}{d} \in \text{GL}_2^{+}(\mathbb{R})$, we define the slash operator in weight $k$ by 
\begin{equation} \label{SLO}
\left(f\mid _k\gamma\right)(z) = \left(\text{det}\gamma\right)^{\frac{k}{2}}(cz+d)^{-k}f(\gamma z).
\end{equation}
With $q = e^{2\pi iz}$, we let $f = \displaystyle{\sum_{n \geq 0}} a(n)q^n \in M_k(N,\chi)$.
For all integers $m\geq 1$, we define the $U$- and $V$-operators by 
\begin{align}
& f\mid_k V_m = m^{-\frac{k}{2}}f\mid_k\pMatrix{m}{0}{0}{1}, \label{VOP} \\
& f\mid_k U_m = m^{\frac{k}{2} - 1}\sum_{j = 0}^{m - 1}f\mid_k \pMatrix{1}{j}{0}{m}. \label{UOP}
\end{align}
On $q$-expansions, we have 
\begin{equation} \label{UVq}
f\mid_k V_m = \sum_n a(n)q^{mn}, \ \ \ 
f\mid_k U_m = \sum_n a(mn)q^n.  
\end{equation}
These operators map spaces of modular forms as follows: 
\begin{align}
& V_m : M_k(N,\chi) \longrightarrow M_k(mN, \chi), \label{vmap} \\
& U_m : M_k(N,\chi)\longrightarrow \begin{cases} M_k(mN, \chi), & m\nmid N, \\ M_k(N, \chi), & m\mid N.\notag
\end{cases}
\end{align}

We will use the following elementary property of the $U$-operator on $q$-series.
\begin{proposition}\label{U_factor}
Let $m\geq 1$ in $\mathbb{Z}$, and let $f(q)$, $g(q)\in \mathbb{C}[[q]]$.  Then we have 
\[
(f(q)g(q^m))\mid U_{m} = f(q)\mid U_{m} \cdot g(q).
\]  
\end{proposition}

\medskip
\noindent
We now introduce Hecke operators. For $f(z) = \displaystyle{\sum_{n \geq 0}} a(n) q^{n} \in M_k(N,\chi)$ and a prime $p$, we define the action of Hecke operator $T_p$ on $f(z)$~by
\begin{equation}\label{defn_Hecke}
f \mid_k T_{p} = f \mid_k U_p + \chi(p) p^{k - 1} f \mid_k V_p = \sum\left(a(p n) + \chi(p){p}^{k - 1}a\left(\frac{n}{p}\right)\right)q^n,
\end{equation}
where $a\left(\frac{n}{p}\right) = 0$ when $p\nmid n$. In general, for a positive integer $n$, we define the action of $T_{n}$ on $f(z)$~by
\begin{equation}\label{defn_Tn}
f \mid_k T_{n} = \sum_{n \geq 0} \left( \sum_{d \mid (m, n)} \chi(d) d^{k - 1} a(mn/d^2) \right) q^{n}.
\end{equation}
The Hecke operators preserve both $M_k(N,\chi)$ and its subspace of Eisenstein and cusp forms. 

\medskip

We call $f \in M_{k}(N,\chi)$ to be a Hecke eigenform if for every integer $n \geq 1$ with $(n, N) = 1$, there exists $\lambda_n \in \C$ with $f \mid_k T_n = \lambda_n f$. Further, $f$ is said to be normalized if $a(1) = 1$.
A newform in $S_{k}^{new}(N, \chi)$ is a normalized cusp form that is an eigenform of all the Hecke operators and all of the Atkin-Lehner operators $W_{p}^{N}$ for primes $p \mid N$ and $H_{N}$. 

\medskip

The Fourier coefficients of normalized Hecke eigenforms satisfy the following multiplicative relations. 
\begin{proposition}
    Let $\displaystyle{f(z) = \sum_{n} a(n)q^{n}} \in M_{k}(N, \chi)$ be a normalized Hecke eigenform. Then for all $m, n \geq 0$, we have
    \begin{equation}\label{multiplicativity}
    a(m)a(n) = \sum_{d \mid (m , n)} \chi(d) d^{k - 1} a(mn/d^{2}).
    \end{equation}
\end{proposition}

It is well known \cite{winnie-li} that every cusp form can be uniquely expressed as a linear combination of newforms acted by $V$-operator.

\begin{proposition}\label{decompose_newform}
Let $f \in S_{k}(N, \chi)$ be a non-zero simultaneous Hecke eigenform for all Hecke operators $T_{n}$ with eigenvalues $a(n)$ for all $(n, N) = 1$. Then there uniquely exist a divisor M of N with $Cond(\chi) \mid M \mid N$ and a newform $h \in S_{k}^{\new}(M, \chi)$ having the same eigenvalues as $f$ for all $T_{n}$ with $(n, N) = 1$, and f can be expressed as a linear combination,
\[
f = \sum_{d \mid \frac{N}{M}} \alpha_{d} \: h \mid_k V_{d}.
\]
\end{proposition}
In 2019, Linowitz and Thompson \cite{lin-lola} showed that an analogous result holds for the decomposition of the Eisenstein subspace.
\begin{theorem} \label{eisen_decomp}
 Let $k \neq 2$, and $f \in E_k(N, \chi)$. Then there uniquely exist a divisor $M$ of $N$ with $Cond(\chi) \mid M \mid N$ and a newform $h \in E_{k}^{\new}(M, \chi)$, and $f$ can be expressed as a linear combination,
\begin{equation}
f = \sum_{d \mid \frac{N}{M}} \alpha_{d} \: h \mid_k V_{d}.
\end{equation}
\end{theorem}
\noindent
We require further operators defined using the slash operator. With 
\begin{align}
H_N = \pMatrix{0}{-1}{N}{0}, \label{Fricke}
\end{align}
we define the Fricke involution $f\mapsto f\mid_k H_N$, which maps $M_{k}(N,\chi)
\rightarrow M_{k}(N,\overline{\chi})$.  
Let $p \mid N$ with $(p, N/p) = 1$. The Atkin-Lehner operator $W_{p}^{N}$ on ${M}_{k}(N, \chi)$ is defined by a matrix of the form, 
\begin{equation}\label{defn_W}
W_{p}^{N} := \pMatrix{pa}{b}{Nc}{p} \:\: \text{where} \:\: a, b, c, d \in \Z, 
\end{equation}
with determinant $p$. If $\chi$ is defined modulo $N/p$, then Lemma 2 of \cite{winnie-li} implies that $\mid_k W_{p}^{N}$ maps $M_{k}(N,\chi)$ to itself and does not depend on the choice of $a, b, c$ and $d$. Let $p$ and $N$ be as defined above.
We define the trace operator on $M_k(N,\chi)$ by
\begin{align}
\text{Tr}_{N/p}^{N}(f) &= f + \overline{\chi}(p) \: {p}^{1 - \frac{k}{2}}f\mid W_{p}^{N} \mid_k U_{p} \label{defn_trace1},  \\
\text{Tr}_{N/p}^{N}(f \mid_k H_{N}) &= f \mid_k H_N + \chi(p) \: {p}^{1 - \frac{k}{2}}f \mid H_N \mid_k W_{p}^{N}\mid U_{p}. \label{defn_trace2}
\end{align}
\noindent
We now provide a definition of newforms in terms of a trace operator.
\begin{lemma}[Lemma 6, \cite{winnie-li}]\label{new_criterion}
Let $p$ and $N$ be as defined above and let $f \in S_{k}(N, \chi)$.  We say $f$ is new at the prime $p$, $f \in S_{k}^{\new p}(N, \chi)$ if and only if $\text{Tr}_{N/p}^{N}(f) = \text{Tr}_{N/p}^{N}(f \mid_k H_N) = 0$.  
\end{lemma}
An analogous criterion for $k \neq 2$ and $f \in E_{k}(N, \chi)$ follows from Proposition 19 of Weisinger \cite{weisinger}.

We will also need the commutation of $V_t$ as in \eqref{VOP} and $W_Q^N$ as in \eqref{defn_W} and also $V_t$ and $H_N$ in a special case.
\begin{proposition}
 Let $t, N \geq 1$ be integers with $t \mid 6$ and $(N, 6) = 1$, let $\chi$ be a Dirichlet character modulo $N$, and let $f \in M_k(N,\chi)$. 
 \begin{enumerate}
 \item With $W_{N}^{6N} = \pmatrix{Na}{b}{6Nc}{Nd}$, we have 
\begin{equation} \label{comm_V_W_6N}
f\mid V_t\mid_k W_{N}^{6N} = \chi\left( \mfrac{6c}{t} \right) \: f\mid_k H_{N}\mid V_t.
\end{equation}
 \item With $H_{6N}= \pmatrix{0}{-1}{6N}{0}$, we have 
 \begin{equation} \label{comm_VH}
 f \mid V_t \mid_k H_{6N} = 6^{k/2} \: t^{-k} \: f \mid H_N \mid_k V_{\mfrac{6}{t}}.
 \end{equation}
 \end{enumerate}
\end{proposition}

Let $\displaystyle{f(z) = \sum_n a(n)q^{n} \in {M}_{k} (N, \chi)}$. We also recall the notion of twisting $f$ by a Dirichlet character $\psi$ modulo $M\geq 1$:  
\begin{equation} \label{twist1}
f \otimes \psi = \sum_n \psi(n)a(n)q^n \in M_k(NM^2, \chi\psi^2).  
\end{equation}
\par
Let $p$ be a prime. We require the commutation of $V_{p}$ and $\psi$ as in \eqref{twist1}.
\begin{equation} \label{Comm_V_twist}
\left( f \otimes \psi  \right) \mid V_{p} = \overline{\psi}(p)  \left( f \mid V_{p} \right) \otimes \psi.
\end{equation}
\noindent
We now proceed to the action of Ramanujan's differential operator. The Ramanujan $\Theta$-operator is defined by: 
    \begin{equation}\label{defn_theta}
    \Theta := q \frac{d}{dq} = \frac{1}{2 \pi i} \frac{d}{dz} \: \text{where} \: q = e^{2 \pi i z}, \: \text{and} \:\:
    \displaystyle{\Theta(f(z)) := \Theta \left(\sum_{n \geq 0} a(n)q^{n} \right) = \sum_{n \geq 0} n a(n) q^{n}}.
    \end{equation}

It can be shown that that the derivative operator $\Theta$ fails to preserve modularity; however, it maps modular forms to quasi-modular forms. 

\begin{lemma}\label{theta_quasi}
Let $\gamma = \pMatrix{a}{b}{c}{d} \in \GL_{2}^{+}(\R)$.
\[
 \Theta(f) \mid_{k + 2}  \gamma =  \Theta(f \mid_k \gamma) + \frac{k}{2 \pi i} \frac{c}{c z + d} \: f \mid_k \gamma.
\]
\end{lemma}
\subsection{Modular forms with theta-multiplier}
The theta function is defined by 
\[
\theta(z) = \sum_{n\in \Z}q^{n^2}. 
\]
This is a modular form of weight $1/2$ on $\Gamma_{0}(4)$ satisfying the following transformation property:
\[
\theta \mid_{1/2} \gamma = \nu_{\theta}(\gamma) \theta(\gamma), \quad \gamma \in \Gamma_{0}(4),
\]
where $\nu_{\theta}$ is the theta-multiplier. For $\gamma = \pMatrix{a}{b}{c}{d} \in \Gamma_{0}(4)$, we have
\[
\nu_{\theta}(\gamma) = \left( \frac{c}{d} \right) \:\:  \epsilon_{d}^{-1} \:\: \text{where} \: 
\ep_d = \begin{cases} 
1\ \ &\text{if}\ \ \   d\equiv 1 \pmod{4},\\
i\ \ \ &\text{if}  \ \ \ d \equiv 3 \pmod{4}.
\end{cases}
\]

Shimura developed a beautiful theory that provides a connection between half-integer weight modular forms and integer weight forms.
\begin{definition}[Shimura Lift]
Let $k \geq 2$ and let $N \geq 1$. Let $\displaystyle{f(z) = \sum_{n = 1}^{\infty}} a(n)q^{n} \in {S}_{k + \frac{1}{2}} (4N, \chi \nu_{\theta}^{2k})$. Let $t$ be a square-free integer and let $\chi_{t}$ be the Dirchlet character modulo $4tN$ such that   
\[
\chi_{t}(m) = \chi(m) \left( \frac{-1}{m} \right) ^{k} \left( \frac{t}{m} \right).
\]
Let $A_{t}(n)$ be the complex numbers defined by
\[
\displaystyle{\sum_{n = 1}^{\infty} A_{t}(n) n^{-s} = L(s - k + 1, \chi_{t}) \left( \sum_{n=1}^{\infty} a(tn^{2}) n^{-s} \right)}.
\]
For all $n \geq 1$, we have:
\[
\displaystyle{A_{t}(n) = \sum_{d \mid n} \chi_{t}(d)  d^{k - 1} a(tn^{2}/d^{2})}.
\]
Then we have $\displaystyle{\textup{Sh}_{t}(f) = \sum_{n = 1}^{\infty} A_{t}(n) q^{n}}  \in S_{2 k}(2N , \chi^{2})$. Furthermore, Hecke equivariance holds.
\[
\textup{Sh}_{t}(f \mid T_{p^2}) = \textup{Sh}_{t}(f) \mid T_{p} \:\: \text{for all primes} \:\: p.
\]
Here, $\textup{Sh}_{t}(f)$ denotes the $t$-th Shimura lift of the function $f$ with respect to theta-multiplier.
\end{definition}

\subsection{Modular forms with eta-multiplier}
The Dedekind eta-function is defined by
\begin{equation}
\eta(z) = q^{1/24}\prod_{n=1}^{\infty}(1 - q^n). \label{eta}
\end{equation}
It is a modular form of weight 1/2 on $\text{SL}_{2}(\mathbb{Z})$ satisfying the following transformation law:
\begin{equation}\label{eq:etamult}
\eta\|_\frac12\gamma=\nu(\gamma) \: \eta \:\: \text{for all} \:\: \gamma\in \Gamma_{0}(1),
\end{equation}
where $\nu_{\eta}$ denotes the eta-multiplier.
We note that $\nu_{\eta}(\gamma) \in \mu_{24}$ for all $\gamma$. The explicit formulas for $\nu_{\eta}(\gamma)$ where $\gamma = \pMatrix{a}{b}{c}{d} \in \Gamma_{0}(1)$ and $c>0$ are given by:
\begin{equation}
\begin{aligned} \label{eq:nueta}
  &\nu_{\eta}(\gamma) = 
  \begin{dcases}
    \( \mfrac dc \) \, e\( \mfrac 1{24} \( (a+d)c-bd(c^2-1)-3c \) \),
  & \text{ if $c$ is odd}, \\
    \(\mfrac cd \) \, e\(\mfrac 1{24} \((a+d)c-bd(c^2-1)+3d-3-3cd\)\),
  & \text{ if $c$ is even,}
  \end{dcases}\\
  &\nu_{\eta}(-\gamma)=i\nu_{\eta}(\gamma).
  \end{aligned}
\end{equation}
The eta-function is a building block for modular forms.  Let $N\geq 1$.  An eta-quotient is a function of the form
\begin{equation}
\label{etaquot}
f(z) = \prod_{\delta \mid N} \eta(\delta z)^{r_{\delta}},
\end{equation}
where $\delta$ and $r_{\delta}$ are integers with $\delta \geq 1$. 
The following proposition (see for example, Proposition 5.9.2 of~\cite{Cohen2017ModularFA}) gives criteria for an eta-quotient to be a weakly holomorphic modular form. 

\begin{proposition}\label{GHN}
Let $f(z)$ be an eta-quotient as in \eqref{etaquot} with $k = \displaystyle{\frac{1}{2} \sum_{\delta\mid N}r_{\delta}\in \Z}$.  Suppose that $f(z)$ satisfies 
\[
\sum_{\delta \mid N}\delta r_{\delta} \equiv 0 \pmod{24} \ \ \text{and} \ \ 
N\sum_{\delta\mid N} \frac{r_{\delta}}{\delta} \equiv 0\pmod{24}.  
\]
Then for all $\gamma = \pMatrix{a}{b}{c}{d} \in \Gamma_0(N)$, we have 
$f(\gamma z) = \chi(d)(cz + d)^k f(z)$, where $\chi$ is defined by $\chi(d) = \left(\frac{(-1)^ks}{d}\right)$ with $s = \displaystyle{\prod_{\delta\mid N}\delta^{r_{\delta}}}$.
\end{proposition}
\noindent
We also require the formula (see for example, Proposition 5.9.3 of \cite{Cohen2017ModularFA}) for the order of vanishing of an eta-quotient at a cusp $\frac{c}{d} \in \mathbb{P}^{1}(\mathbb{Q})$.

\begin{proposition} \label{CE}
Let $N$, $c$, and $d$ be positive integers with $\gcd(c, d) = 1$,
and let $f(z)$ be a level $N$ eta-quotient as in \eqref{etaquot}.  Then the order of vanishing of $f(z)$ at the cusp $\frac{c}{d}$ is given by
\[
\textup{ord}_{\frac{c}{d}}(f) = \frac{1}{24} \sum_{\delta \mid N} \frac{(\textup{gcd}(d, \delta))^{2} \: r_{\delta}}{\delta}.
\]
\end{proposition}
\noindent
We note that, with respect to $\Gamma_0(N)$, the width of the cusp $\frac{c}{d}$ is $h_{d, N} = \frac{N}{\gcd(d^2, N)}$.  Therefore, the order of vanishing of the eta-quotient $f(z)$ in the local variable $q_{h_{d, N}} = e^{\frac{2\pi iz}{h_{d, N}}}$ is $h_{d, N}\cdot\text{ord}_{\frac{c}{d}}(f)$. We also observe that $\textup{ord}_{\frac{c}{d}}(f) = \textup{ord}_{\frac{1}{d}}(f)$. Therefore, the eta-quotient is holomorphic if and only if for all $d \mid N$, we have $\textup{ord}_{\frac{1}{d}}(f)  \geq 0$. We also define the denominator $D$ to be the least positive integer such that for all $d \mid N$, we have
\begin{equation}
D \cdot ord_{\mfrac{1}{d}}(f) \in \Z.
\end{equation}

We turn our attention to forms $f$ that transform with regards to a power of eta-multiplier times a Dirichlet character. We begin by studying the action of $U_p$ and $V_p$ operator on these forms. 
Let $(r, 24) = 1$ and let $f = \displaystyle{\sum_{n\equiv r\spmod{24}}} a(n) q^{n} \in M_{k}(N,\chi\nu_\eta^r)$, and that $m$ is a positive integer. 
Define 
\begin{align*}
  f \mid U_m:=\sum a(mn)q^\frac n{24}
\quad {and}\quad
  f \mid V_m:=\sum a(n)q^\frac {mn}{24}.
\end{align*}
\noindent
For a prime $p \geq 5$, and for a character modulo $p$, $\chi_{p} = \left( \frac{\cdot}{p} \right)$, Lemma 2.1 of \cite{ABR} implies that
\begin{equation}\label{action_UV}
  U_p \,:\, M_k\left(N, \chi \nu_\eta^r\right) \longrightarrow M_k\left(N \mfrac p{(N, p)},\, \chi\chi_p\nu_\eta^{pr}\right) \:\: \text{and} \:\:
  V_p \,:\, M_k(N, \chi\nu_\eta^r) \longrightarrow M_k\left(Np, \,\chi\chi_p\nu_\eta^{pr}\right).
\end{equation}
Our next lemma encodes the action of the $U_{p}$ and $V_{p}$ operator for a prime $p \in \{2,3\}$. We recall that $k \in \Z$ if and only if $r$ is even and $k \in \Z + \frac{1}{2}$ if and only if $r$ is odd.
\begin{lemma} \label{action_U_V}
Let $r \in \Z$ and $\chi \in \hat{G}_{N}$.
\begin{enumerate}
\item Let $f \in M_{k}^{!}(N, \chi \nu_{\eta}^{r})$. 
 \begin{enumerate}
 \item If $3 \mid r$, then we have
 \begin{equation}
\begin{aligned}
  &f \mid_k U_{3} = 
  \begin{dcases}
    M_{k}^{!}\left(3N, \chi \left( \frac{\cdot}{3} \right)^r \nu_{\eta}^{3r}\right),
  & \text{$3 \nmid N$}, \\
  M_{k}^{!}\left(N, \chi \left( \frac{\cdot}{3} \right)^r \nu_{\eta}^{3r}\right),
  & \text{$3 \parallel N$}, \\
    M_{k}^{!}\left(N/3, \: \chi \left( \frac{\cdot}{3} \right)^r \nu_{\eta}^{3r}\right),
  & \text{$3^2 \mid N$}.
\end{dcases}
\end{aligned}
\end{equation}
Furthermore, if $6 \mid r$ and $3 \nmid N$, then we have $f \mid_k U_{3} \in M_{k}^{!}(3N, \chi \nu_{\eta}^{3r})$.
\item  Let $f \in M_{k}^{!}(N, \chi \nu_{\eta}^{r})$. We have
 \begin{equation}
\begin{aligned}
  &f \mid_k V_{3} = 
  \begin{dcases}
    M_{k}^{!}\left(3N, \chi \left( \frac{\cdot}{3} \right)^r \nu_{\eta}^{3r}\right),
  & \text{$3 \mid N$ or $3 \mid r$}, \\
  M_{k}^{!}\left(9N, \chi \left( \frac{\cdot}{3} \right)^r \nu_{\eta}^{3r}\right),
  & \text{$3 \nmid N$ and $3 \nmid r$}.
\end{dcases}
\end{aligned}
\end{equation}    
\end{enumerate}
\item Let $f \in M_{k}^{!}(N, \chi \nu_{\eta}^{2r})$. 
\begin{enumerate}
\item We have  $f \mid_k V_{2} \in M_{k}^{!}\left( \frac{8N}{\gcd(4, N)}, \left( \frac{-4}{\cdot} \right)^{r} \chi \nu_{\eta}^{4r}  \right)$.
 \item Let $t \geq 1$. We have $\left(\eta^2 \mid_1 V_{2^{t}} \right)(z) \in S_{1}\left( 4 \cdot 2^{t}, \left( \frac{-4}{\cdot} \right) \nu_{\eta}^{2 \cdot 2^t} \right)$ and 
\begin{equation}
\begin{aligned}
  &\left(\eta \mid_{1/2} V_{2^{t}}\right)(z) = 
  \begin{dcases}
   S_{1/2}\left( 16, \left( \frac{-4}{\cdot} \right) \nu_{\theta} \cdot \nu_{\eta}^{2} \right),
  & \text{$t = 1$}, \\
    S_{1/2}\left( 8 \cdot 2^{t}, \left( \frac{8}{\cdot} \right)^{t - 1} \nu_{\theta} \cdot \nu_{\eta}^{2^t} \right),  & \text{$t \geq 2$}.
\end{dcases}
\end{aligned}
\end{equation}
 \end{enumerate} 
 \item Let $k \in \Z$ and let $f \in M_{k}^{!}(N, \chi \nu_{\eta}^{r})$. We have
 \begin{equation}
\begin{aligned}
  &f \mid_k U_{2} = 
  \begin{dcases}
    M_{k}^{!}\left(8N, \chi \left( \frac{-4}{\cdot} \right)^{r/2} \nu_{\eta}^{2r}\right),
  & \text{$2  \parallel r$}, \\
  M_{k}^{!}\left(4N, \chi \nu_{\eta}^{2r}\right),
  & \text{$2^2 \mid r$}, \\
    M_{k}^{!}\left(2N, \: \chi \nu_{\eta}^{2r}\right),
  & \text{$4^2 \mid r$}.
\end{dcases}
\end{aligned}
\end{equation}
 \end{enumerate}
\end{lemma}
\begin{proof}
  We prove (1.a) of Lemma \ref{action_U_V}; we use a similar approach to prove (1.b), (2.a), (2.b) and (3).
  
  To begin with, we consider $r \in \Z$ be odd with $3 \mid r$, and let $k \in \Z + \frac{1}{2}$. Let $\chi \in \hat{G}_{N}$ and let $f(z) \in  M_{k}^{!}(N, \chi \nu_{\eta}^{r})$. Using Proposition \ref{GHN} and Proposition \ref{CE}, we have $\frac{\eta(z)^9}{\eta(3z)^3} \in M_{3}\left( 3, \left( \frac{\cdot}{3} \right) \right)$  and $\eta(z)^9 \in S_{9/2}(1, \nu_{\eta}^9)$ which implies that $\eta(3z)^3 \in S_{3/2}\left( 3, \left( \frac{\cdot}{3} \right)\nu_{\eta}^9  \right)$. It follows that $\eta(3z)^{3r} \in S_{3r/2}\left( 3, \left( \frac{\cdot}{3} \right)^{r} \nu_{\eta}^r  \right)$. We compute
  \begin{align*}
      \frac{f(z)}{\eta^{3r}(3z)} &\in M_{k - \left(\frac{3r - 1}{2} \right)}^{!}\left( \mfrac{3N}{(N, 9)}, \chi \left( \frac{\cdot}{3} \right)^{r} \right), \\
      \frac{f(z)}{\eta^{3r}(3z)} \mathrel{\bigg|} U_3 &= \frac{(f \mid U_{3})(z)}{\eta^{3r}(z)} \in M_{k - \left(\frac{3r - 1}{2} \right)}^{!}\left( \mfrac{3N}{(N, 9)}, \chi \left( \frac{\cdot}{3} \right)^{r} \right), \:\: \text{using Proposition \ref{U_factor}} \\
       (f \mid U_{3})(z) &\in M_{k + \frac{1}{2}}^{!}\left(\mfrac{3N}{(N, 9)}, \chi \left( \frac{\cdot}{3} \right)^{r} \nu_{\eta}^{3r} \right).
\end{align*}
Next, we consider the case when $6 \mid r$. Let $k \in \Z$. From earlier work, it follows that $\eta(3z)^{r} \in S_{r/2}\left(3, \nu_{\eta}^{3r} \right)$. We compute
\begin{align*}
    \eta(3z)^r f(z) &\in M_{k + \frac{r}{2}}^{!}\left( 3N, \chi \right) \:\: \text{since} \:\: 4r \equiv 0 \spmod{24}, \:\: \\
    \eta(z)^r ((f \mid U_{3})(z)) &\in M_{k + \frac{r}{2}}^{!}\left( 3N, \chi \right), \:\: \text{using Proposition \ref{U_factor}} \\
    (f \mid U_{3})(z) &\in M_{k}^{!}\left(3N, \chi \nu_{\eta}^{3r} \right).
\end{align*}
\end{proof}
Let $f$ be as defined above and let $p \geq 5$ prime. We next recall the notion of twisting $f$ by $\chi_p$ as follows:
\begin{equation} \label{twist}
f \otimes \chi_p = \sum \chi_p(n)a(n)q^n.  
\end{equation} 
By \S 2 of \cite{ABR}, it follows that $f \otimes \chi_{p} \in M_{k}^{!}(N p^2, \chi \nu_{\eta}^r)$.

\medskip

In 1991, Mersmann \cite{mersmann} showed that there exist finitely many holomorphic eta-quotients that are spanned by theta functions. His result was refined by  
Oliver \cite{oliver} who provided a complete classification of eta-quotients that are theta functions as defined in \eqref{gen_theta} associated with the theta multiplier. In the following lemma, we provide a corresponding classification for eta-quotients that are theta functions with the eta-multiplier. 

\begin{lemma} 
\begin{enumerate} \label{theta_series}
\item Let $\chi$ be an even Dirichlet character modulo $N$ and let $D$ denote the denominator of the eta-quotient.
    \begin{equation} \label{theta_even}
    \Theta_{\chi}(z) = \sum_{n \geq 1} \chi(n) q^{\frac{n^2}{D}} \in M_{1/2}(N, \psi \nu_{\eta}^r). 
    \end{equation}
    \begin{enumerate}
  \item  We provide a list of theta functions of the form \eqref{theta_even} associated with the eta-multiplier. All of them are expressible as eta-quotients.
\renewcommand{\arraystretch}{1.5} 
\setlength{\tabcolsep}{12pt} 
\begin{table}[ht]
  \centering
\begin{tabular}{|c|c|c|c|c|c|}
\hline
$D$ & $\chi$ & $N$ & $\psi$ & $r$ & \text{eta-quotient} \\ 
\hline
8 & $1_2$ & 8 & $\left(\mfrac{-4}{\cdot} \right)$ & 3 & $\mfrac{\eta^{2}(2z)}{\eta(z)}$ \\
 \hline
8 & $\left(\mfrac{8}{\cdot}\right)$ & 32 & $\left(\mfrac{-8}{\cdot}\right)$ & 3 & $\mfrac{\eta(z) \eta(4z)}{\eta(2z)}$ \\
 \hline
24 & $\left(\mfrac{12}{\cdot}\right)$ & 1 & $1_1$ & 1 &  $\eta(z)$\\
 \hline
24 & $\left(\mfrac{24}{\cdot}\right)$ & 32 & $\left(\mfrac{8}{\cdot}\right)$ & 1 & $\mfrac{\eta^3(2z)}{\eta(z) \eta(4z)}$ \\
\hline
24 & $1_6$ & 144 & $\left(\mfrac{12}{\cdot}\right)$ & 1 & $\mfrac{\eta(2z) \eta^2(3z)}{\eta(z) \eta(6z)}$  \\
\hline
24 & $\left(\mfrac{8}{\cdot}\right) 1_{3}$ & 288 & $\left(\mfrac{24}{\cdot}\right)$ & 1 & $\mfrac{\eta(z) \eta(4z) \eta^5(6z)}{\eta^2(2z) \eta^2(3z) \eta^2(12z)}$\\
 \hline
\end{tabular}
  \caption{Parameter Values in \eqref{theta_even} for $N, \chi$ and $r$.}
\end{table}
\end{enumerate}
 \item Let $\chi$ be an odd Dirichlet character modulo $N$ and let $D$ denote the denominator of the eta-quotient.
    \begin{equation} \label{theta_odd}
    \Theta_{\chi}(z) = \sum_{n \geq 1} \chi(n) n q^{\frac{n^2}{D}} \in M_{3/2}(N, \psi \nu_{\eta}^r). 
    \end{equation}
    \begin{enumerate}
  \item  We provide a list of theta functions of the form \eqref{theta_odd} associated with the eta-multiplier. The first four in the list are expressible as eta-quotients while the last two lie outside this class.
  
 \renewcommand{\arraystretch}{1.8} 
\setlength{\tabcolsep}{12pt} 
\begin{table}[ht]
  \centering
\begin{tabular}{|c|c|c|c|c|c|}
\hline
$D$ & $\chi$ & $N$ & $\psi$ & $r$ & \text{eta-quotient} \\ 
\hline
8 & $\left(\mfrac{-4}{\cdot}\right)$ & 1 & $1_1$ & 3 & $\eta^{3}(z)$ \\
 \hline
8 & $\left(\mfrac{-8}{\cdot}\right)$ & 32 & $\left(\mfrac{8}{\cdot}\right)$ & 3 & $\mfrac{\eta^9(2z)}{\eta^{3}(z) \eta^3(4z)}$ \\
 \hline
24 & $\left(\mfrac{-3}{\cdot}\right) 1_{2}$ & 8 & $\left(\mfrac{-4}{\cdot}\right)$ & 1 & $\mfrac{\eta^5(z)}{\eta^{2}(2z)}$ \\
 \hline
24 & $\left(\mfrac{-24}{\cdot}\right)$ & 32 & $\left(\mfrac{-8}{\cdot}\right)$ & 1 & $\mfrac{\eta^{13}(2z)}{\eta^{5}(z) \eta^5(4z)}$ \\
 \hline
8 & $\left(\mfrac{-4}{\cdot}\right) 1_{3}$ & 9 & $1_1$ & 3 & \\
 \hline
8 & $\left(\mfrac{-8}{\cdot}\right) 1_{3}$ & 288 & $\left(\mfrac{8}{\cdot}\right)$ & 3 & \\ 
\hline 
\end{tabular}
  \caption{Parameter Values in \eqref{theta_odd} for $N, \chi$ and $r$.}
    \label{tab:eta_odd}
\end{table}
\end{enumerate}
\end{enumerate}
\end{lemma}

\newpage

{\bf{Remarks:}}
\begin{enumerate} [label=(\alph*)]
    \item  We note that theta series corresponding to $\chi = \left( \mfrac{-4}{\cdot} \right) 1_3$ and $\chi = \left( \mfrac{-8}{\cdot} \right) 1_3$ are not eta-quotients but arise as a linear combination of eta products as follows:
     \begin{align} 
     \sum_{n \geq 1} \left( \mfrac{-4}{n} \right) \: 1_3(n) \: n q^{\mfrac{n^2}{8}} &= \sum \left( \mfrac{-4}{n} \right) \: n q^{\mfrac{n^2}{8}} + 3 \sum \left( \mfrac{-4}{n} \right) \: n q^{\mfrac{n^2}{8}} \mathrel{\bigg|} V_9, \\
     &= \eta^3(z) + 3 \cdot \eta^3(9z) \in S_{3/2}(9, \nu_{\eta}^{3}). \label{eta_prod1} \\
     \sum_{n \geq 1} \left( \mfrac{-8}{n} \right) \: 1_3(n) \: n q^{\mfrac{n^2}{8}} &= \sum \left( \mfrac{-8}{n} \right) \: n q^{\mfrac{n^2}{8}} - 3 \sum \left( \mfrac{-8}{n} \right) \: n q^{\mfrac{n^2}{8}} \mathrel{\bigg|} V_9, \\
    & = \mfrac{\eta^9(2z)}{\eta^3(z) \eta^3(4z)} - 3 \mfrac{\eta^9(18z)}{\eta^3(9z) \eta^3(36z)}  \in M_{3/2}\left(288, \left( \mfrac{8}{\cdot} \right) \nu_{\eta}^3 \right). \label{eta_prod2}
     \end{align}
     \item The theta series corresponding to $D = 3, \chi = 1_3, N = 288$ and $\psi = \left( \mfrac{12}{\cdot} \right)$ in even case and $D = 3, \chi = \left( \mfrac{-3}{\cdot} \right), N = 16$ and $\psi = \left( \mfrac{-4}{\cdot} \right)$ in odd case has multiplier $\nu_{\theta}\nu_{\eta}^8$.
\end{enumerate}

\medskip

In a recent work, Ahlgren et. al. \cite{AAD} proved a Shimura type Correspondence for forms with the eta multiplier. 
Let $(N, 6) = 1$ and let $1 \leq r \leq 23$ with $(r, 6) = 1$. Let $k \geq 0, \: \chi \in \widehat{G}_{N}$ and $t \geq 1$ be square-free.
\begin{theorem}[Theorem 1, \cite{AAD}] \label{AAD1}
Let $f(z) = \displaystyle{\sum_{n\equiv r\spmod{24}}} a(n) q^\frac n{24} \in S_{k + 1/2}{(N, \chi \nu_{\eta}^{r})}$ where $k \geq 2$. For a square-free $t$, we define $B_{t}(n)$ by 
\[
 \sum_{n \geq 1} \frac{B_{t}(n)}{n^{s}} = L \left(s - k + 1, \chi \left( \frac{\cdot}{t} \right) \right) \sum_{n \geq 1} \left( \frac{12}{n} \right) \frac{a(t n^{2})}{n^{s}} 
 \]
 For all $n \geq 1$, we have:
\begin{equation} \label{image_24}
\displaystyle{B_{t}(n) = \sum_{d \mid n} \chi(d) \left( \frac{d}{t} \right) \left( \frac{12}{n/d} \right) d^{k - 1} a(tn^{2}/d^{2})}.
\end{equation}
Then we have $\mathcal{S}_{t}(f) = \displaystyle{\sum_{n = 1}^{\infty}} B_{t}(n) q^{n} \in S_{2 k}(6N, \chi^{2}, \epsilon_{2, r, \chi}, \epsilon_{3, r, \chi})$.
\end{theorem}
\noindent
Here, $\mathcal{S}_{t}(f)$ denotes the $t$-th Shimura lift of the function $f$ with respect to eta-multiplier.
\begin{remark}
We note that $\mathcal{S}_{t}(f) = 0$ unless $t \equiv r \spmod{24}$. Furthermore, if we remove the hypothesis $(N, 6) = 1$, then $\mathcal{S}_{t}(f) \in S_{2 k}(6N, \chi^{2})$.
\end{remark}
 \begin{lemma}[Commutativity relations] \label{Comm_(r,6)=1}
Let $(r,6) = 1$. Let $f \in \mathcal{S}_{k + \frac{1}{2}} (N, \chi \nu_{\eta}^{r})$ and let $t \geq 1$ be square-free with $t \equiv r \spmod{24}$. Let $p \geq 5$ be prime.
 \begin{enumerate} [label=(\alph*)]
     \item $\mathcal{S}_{t} (f \mid T_{p^{2}}) = \left( \frac{12}{p} \right) \: \mathcal{S}_{t}(f) \mid T_{p}$. \label{eq:Comm_S1:a}
     \item $\mathcal{S}_{1}(f \mid V_{t}) = \left( \frac{12}{t} \right) \: \mathcal{S}_{t}(f) \mid V_{t}$. \label{eq:Comm_S1:b} 
     \item $\mathcal{S}_{1}(f \mid V_{p^2}) = \left( \frac{12}{p} \right) \: \mathcal{S}_{1}(f) \mid V_{p}$. \label{eq:Comm_S1:c}
     \item $\mathcal{S}_{t}(f) = \mathcal{S}_{1}(f \mid U_{t})$. \label{eq:Comm_S1:d}
 \end{enumerate}
 \end{lemma}
 \begin{proof}
 For proof of part \ref{eq:Comm_S1:a}, we refer the reader to \S5 and \S6 of \cite{AAD}.
We proceed to prove part \ref{eq:Comm_S1:b}. Let $f \in S_{k + 1/2}(N, \chi \nu_{\eta}^r)$. This implies that  $f(z) = \displaystyle{\sum_{n\equiv r\spmod{24}}} a(n) q^\frac n{24}$.  Using \eqref{action_UV}, it follows that 
 \[
 f \mid V_{t} = \sum_{\substack{m \equiv 1 \spmod{24} \\ t \mid m}} b_{t}(m) q^{\frac{m}{24}} \in S_{k + 1/2}\left( Nt, \left( \frac{\cdot}{t} \right) \chi \nu_{\eta} \right) \:\: \text{where} \:\: b_{t}(m) = a\left( \frac{m}{t} \right). 
\]
We compute
\begin{align}
 \Sh_{1}(f \mid V_{t}) = \Sh_{1}\left( \sum_{\substack{m \equiv 1 \spmod{24} \\ t \mid m}} b_{t}(m) q^{\frac{m}{24}} \right) &= \sum_{n} \left( \sum_{d \mid n} \chi(d) d^{k - 1} \left( \frac{d}{t} \right) \left( \frac{12}{n/d} \right) \: b_{t}\left(\frac{n^{2}}{d^{2}}\right) \right) q^{n}, \\
 &= \sum_{n, \:t \mid n} \left( \sum_{d \mid \frac{n}{t}} \chi(d) d^{k - 1} \left( \frac{d}{t} \right) \left( \frac{12}{n/d} \right) \: a\left(\frac{n^{2}}{td^{2}}\right) \right) q^{n}. \label{LHS}
\end{align}
On the other hand, 
\begin{align}
\left( \frac{12}{t} \right) \: \mathcal{S}_{t}(f) \mid V_{t} &= \sum_n \left( \sum_{d \mid n} \chi(d) d^{k - 1} \left( \frac{d}{t} \right) \left( \frac{12}{t} \right) \left( \frac{12}{n/d} \right) \: a\left(\frac{tn^{2}}{d^{2}}\right) \right) q^{tn}, \:\: \text{using \eqref{image_24}} \\
&= \sum_{s, \: t \mid s} \left( \sum_{d \mid \frac{s}{t}} \chi(d) d^{k - 1} \left( \frac{d}{t} \right) \left( \frac{12}{s/d} \right) \: a\left(\frac{s^{2}}{td^{2}}\right) \right) q^{s}, \:\: \text{by plugging in} \:\: s = tn. \label{RHS}
\end{align}
Comparing \eqref{LHS} and \eqref{RHS}, we obtain the desired result.

\medskip

We next prove part \ref{eq:Comm_S1:c}. Let $f \in S_{k + 1/2}(N, \chi \nu_{\eta})$ and let $p \geq 5$ prime. This implies that  $f(z) = \displaystyle{\sum_{n\equiv 1 \spmod{24}}} a(n) q^\frac n{24}$.  Using \eqref{action_UV}, it follows that 
 \[
 f \mid V_{p^{2}} = \sum_{\substack{m \equiv 1 \spmod{24} \\ p^2 \mid m}} b_{p^2}(m) q^{\frac{m}{24}} \in S_{k + 1/2}\left( Np^2, \chi \nu_{\eta} \right) \:\: \text{where} \:\: b_{p^2}(m) = a\left( \frac{m}{p^2} \right). 
\]
We compute
\begin{align*}
 \Sh_{1}(f \mid V_{p^2}) = \Sh_{1}\left( \sum_{\substack{m \equiv 1 \spmod{24} \\ p^2 \mid m}} b_{p^2}(m) q^{\frac{m}{24}} \right) &= \sum_{n} \left( \sum_{d \mid n} \chi(d) d^{k - 1} \left( \frac{12}{n/d} \right) \: b_{p^2}\left( \frac{n^2}{d^2} \right) \right) q^{n}, \\
 &= \sum_{n, \:p \mid n} \left( \sum_{d \mid \frac{n}{p}} \chi(d) d^{k - 1} \left( \frac{12}{n/d} \right) \: a\left(\frac{n^{2}}{p^2 d^{2}}\right) \right) q^{n}, \\
 &= \sum_{m} \left( \sum_{d \mid m} \chi(d) d^{k - 1} \left( \frac{12}{mp/d} \right) \: 
 a\left(\frac{m^{2}}{d^{2}}\right) \right) q^{pm}, \:\: \text{plug in} \:\: n = m p.\\
 &= \left( \frac{12}{p} \right) \: \Sh_{1}(f) \mid V_{p}.
\end{align*}
The proof of part \ref{eq:Comm_S1:d} is obtained using similar approach as above.
 \end{proof}
 \noindent
 An analogous version of Theorem \ref{AAD1} exists for $(r, 6) = 3$.
 \begin{theorem}[Theorem 2, \cite{AAD}]\label{AAD2}
Let $f(z) = \displaystyle{\sum_{n\equiv r/3 \spmod{8}}} b(n) q^\frac n{8} \in S_{k + 1/2}{(N, \chi \nu_{\eta}^{r})}$ where $k \geq 2$ and $N \geq 1$ odd. For a square-free $t$, we define $B_{t}(n)$ by 
\[
 \sum_{n \geq 1} \frac{B_{t}(n)}{n^{s}} = L \left(s - k + 1, \chi \left( \frac{\cdot}{t} \right) \right) \sum_{n \geq 1} \left( \frac{-4}{n} \right) \frac{a(t n^{2})}{n^{s}}. 
 \]
 For all $n \geq 1$, we have:
\begin{equation} \label{image_8}
\displaystyle{B_{t}(n) = \sum_{d \mid n} \chi(d) \left( \frac{d}{t} \right) \left( \frac{-4}{n/d} \right) d^{k - 1} a(tn^{2}/d^{2})}.
\end{equation}
Then we have $\mathcal{S}_{t}(f) = \displaystyle{\sum_{n = 1}^{\infty}} B_{t}(n) q^{n} \in S_{2 k}(2N, \chi^{2}, \epsilon_{2, r, \chi})$.
\end{theorem}
\begin{remark}
We note that $\mathcal{S}_{t}(f) = 0$ unless $t \equiv r/3 \spmod{8}$. Furthermore, if we remove the hypothesis $N$ is odd, then $\mathcal{S}_{t}(f) \in S_{2 k}(2N, \chi^{2})$.
\end{remark}
 \begin{lemma}[Commutativity relations] \label{Comm_r/3}
 Let $(r,6) = 3$. Let $f \in \mathcal{S}_{k + \frac{1}{2}} (N, \chi \nu_{\eta}^{r})$ and let $t \geq 1$ be square-free with $t \equiv r/3 \spmod{8}$. Let $p \geq 3$ be prime.
 \begin{enumerate}[label=(\alph*), ref=\thelemma.\alph*]
  \item $\mathcal{S}_{t} (f \mid T_{p^{2}}) = \left( \frac{-4}{p} \right) \: \mathcal{S}_{t}(f) \mid T_{p}$. \label{eq:Comm_r/3:a}
     \item $\mathcal{S}_{1}(f \mid V_{t}) = \left( \frac{-4}{t} \right) \: \mathcal{S}_{t}(f) \mid V_{t}$. \label{eq:Comm_r/3:b}
     \item $\mathcal{S}_{1}(f \mid V_{p^2}) = \left( \frac{-4}{p} \right) \: \mathcal{S}_{1}(f) \mid V_{p}$. \label{eq:Comm_r/3:c}
     \item $\mathcal{S}_{t}(f) = \mathcal{S}_{1}(f \mid U_{t})$. \label{eq:Comm_r/3:d}
 \end{enumerate}
 \end{lemma}
We note that the proof of Lemma \ref{Comm_r/3} follows a similar approach as in Lemma \ref{Comm_(r,6)=1}.
\medskip

 Next, we state some important facts on modular forms that are eigenforms for $T_{p^{2}}$ for all primes $p$.

 \begin{proposition} 
     Let $t, N \geq 1$ with $t$-square free, let $k \geq 2$, and let $(r, 6) = 1$. Suppose, for all primes $p$, that $\displaystyle{f(z) = \sum_{\substack{n \geq 1 \\ n \equiv r \spmod{24}}} a(n)q^{\frac{n}{24}} \in {S}_{k + 1/2}(N, \chi \nu_{\eta}^{r})}$ is an eigenform for $T_{p^{2}}$ with eigenvalue $\gamma_{p}$. Then, for all square-free $t_{1},t_{2} \geq 1$, we have
     \[
     a(t_{1}) \: \mathcal{S}_{t_{2}}(f) = a(t_{2}) \: \mathcal{S}_{t_{1}}(f).
     \]
\end{proposition}
\noindent
We note that analogous result holds in case of $(r,6) = 3$.

\noindent
The following result provides a connection between the lifts $\textup{Sh}_{t}(f)$ and $\mathcal{S}_{t}(f)$.
\begin{proposition}[\S 3.4, \cite{AAD}] 
Let $f \in S_{k + 1/2}(N, \chi \nu_{\eta}^{r})$ and $t \geq 1$ be a square-free integer.
\begin{enumerate}
\item Let $(r,6) = 1$. 
\[
\textup{Sh}_{t}(f(24z)) = \mathcal{S}_{t}(f) \otimes \left( \frac{12}{\cdot} \right).
\]
\item Let $(r,6) = 3$. 
\[
\textup{Sh}_{t}(f(8z)) = \mathcal{S}_{t}(f) \otimes \left( \frac{-4}{\cdot} \right).
\]
\end{enumerate}
\end{proposition}
\section{Proofs}
\subsection{Proof of Theorem \ref{first_Shimura_image}}
We prove part (1.e) of Theorem \ref{first_Shimura_image} and the proof of remaining cases follow a similar approach.

\medskip

Let $1 \leq r \leq 23$ with $(r, 6) = 1$. Let $\lambda \geq 2$ and $N \geq 1$ be integers. Let $t \geq 1$ be a square-free integer and let $\chi \in \hat{G}_{N}$. We consider $g(z) = \displaystyle{\sum_{m \geq 0}} a(m) q^{m} \in M_{k}(N, \chi)$ to be a normalized Hecke eigenform for all $T_p$. Using part (1.a) of Lemma \ref{theta_series}, it follows that
\begin{equation}
\mfrac{\eta^3(2z)}{\eta(z) \eta(4z)} g(z) =  \sum_{n \geq 1} \sum_{m \geq 0} \left( \frac{24}{n} \right) \: a(m) \: q^{m + \mfrac{n^2}{24}} \in M_{k + 1/2}\left(\frac{32N}{(N, 32)}, \chi \left(  \frac{8}{\cdot} \right) \nu_{\eta}  \right).   
\end{equation}
Plugging in $\mfrac{t}{24} = m + \frac{n^2}{24}$ implies that $t \equiv 1 \spmod{24}$. Therefore, we have
\begin{equation}
\mfrac{\eta^3(2z)}{\eta(z) \eta(4z)} g(z) = \sum_{\substack{t \geq 1 \\ t \equiv 1 \spmod{24}}} \left( \sum_{n \geq 1} \left(\frac{24}{n} \right) a \left( \frac{t - n^2}{24} \right)  \right)   \: q^{\frac{t}{24}} = \sum_{\substack{t \geq 1 \\ t \equiv 1 \spmod{24}}} b(t) \: q^{\frac{t}{24}}, \:\: \text{where} 
 \end{equation}
 \begin{equation}
\begin{aligned}
  b(t) = 
  \begin{dcases}
   \sum_{n \geq 1} \left(\frac{24}{n} \right) a \left( \frac{t - n^2}{24} \right), & \text{$t \equiv 1 \spmod{24}$}
   \\
    0,  & \text{$t \not\equiv 1 \spmod{24}$}.
\end{dcases}
\end{aligned}
\end{equation}
Using Theorem \ref{AAD1}, we have
\begin{equation}
\mathcal{S}_{1}\left(\frac{\eta^3(2z)}{\eta(z) \eta(4z)} g(z)\right) = \mathcal{S}_{1}\left( \sum_{\substack{m \geq 1 \\ m \equiv 1 \spmod{24}}} b(m) q^{\frac{m}{24}} \right)   = \sum_{n \geq 1} B_{1}(n) q^{n},
\end{equation}
where \eqref{image_24} gives
\begin{align}
B_{1}(n) &= \sum_{d \mid n} \chi(d) \left( \frac{8}{d} \right) \left( \mfrac{12}{n/d} \right) d^{k - 1} b \left( \mfrac{n^2}{d^2}  \right),   \\
&= \sum_{d \mid n} \chi(d) \left( \frac{8}{d} \right) \left( \mfrac{12}{n/d} \right) d^{k - 1} \sum_{m \geq 1} \left(  \left(  \mfrac{24}{m} \right) a \left( \mfrac{\mfrac{n^2}{d^2} - m^2}{24}  \right) \right), \\
&= \sum_{d \mid n} \chi(d) \left( \frac{8}{d} \right) \left( \mfrac{12}{n/d} \right) d^{k - 1} \sum_{m \geq 1}  \left(  \mfrac{24}{m} \right) a(n, d, m) \label{exp_A1}, 
\end{align}
where we have written $a(n, d, m):= a \left(\mfrac{\mfrac{n^2}{d^2} - m^2}{24}\right)$ for brevity. We now require the following lemma:
\begin{lemma}\label{simplify_sum}
With notations and definitions as above, we have    
\begin{align}
 \sum_{m \geq 1}  \left(  \mfrac{24}{m} \right) a(n, d, m) = \left( \mfrac{24}{n/d}  \right) \left( \sum_{\substack{r \geq 0 \\ r \equiv 0 \spmod{2} \\ d \mid r}} a\left( \mfrac{r(n - 6r)}{d^2} \right) - \sum_{\substack{s \geq 0 \\ s \equiv 1 \spmod{2} \\ d \mid s}} a\left( \mfrac{s(n - 6s)}{d^2} \right) \right) \\
 + \left( \mfrac{24}{n/d}  \right) \left( \mfrac{-4}{n}  \right) \left( \sum_{\substack{t \geq 0 \\ t \equiv 1 \spmod{4} \\ d \mid t}} a\left( \mfrac{t \left(\mfrac{n - 3t}{2}\right)}{d^2} \right)  - \sum_{\substack{u \geq 0 \\ u \equiv 3 \spmod{4} \\ d \mid u}} a\left( \mfrac{u\left(\mfrac{n - 3u}{2}\right)}{d^2} \right)  \right).
\end{align}
\end{lemma}
\begin{proof}
We note that due to $\left( \mfrac{24}{m} \right)$ factor in each term, the sum on the left is supported on integers $m$ with $(m, 24) = 1$. For such $m$, we have $m^2 \equiv 1 \spmod{24}$. 
Furthermore, we have
\begin{equation}
b \left( \mfrac{n^2}{d^2} \right)  \neq 0 \iff \mfrac{n^2}{d^2} \equiv 1 \spmod{24}.   
\end{equation}
This implies that $\left( \mfrac{n}{d}, 6  \right) = 1$. Therefore, we fix $n \geq 1$ and for all $d \mid n$ with $\left( \mfrac{n}{d}, 6  \right) = 1$, we deduce that
\begin{equation}
m \in \left\{\pm \mfrac{n}{d}, \:\: \pm \left(\mfrac{n}{d} - 12 \right), \:\: \pm \left(\mfrac{n}{d} - 6\right), \:\: \pm \left(\mfrac{n}{d} - 18 \right)  \right\} \spmod{24}.  
\end{equation}

\begin{align}
 \sum_{m \geq 1}  \left(  \mfrac{24}{m} \right) a(n, d, m) &= \left( \mfrac{24}{n/d} \right) \left( \sum_{m \equiv \pm \mfrac{n}{d} \spmod{24}} a(n, d, m) - \sum_{m \equiv \pm \left(\mfrac{n}{d} - 12 \right) \spmod{24}} a(n, d, m) \right) \\
& + \left( \mfrac{-24}{n/d} \right) \left( \sum_{m \equiv \pm \left( \mfrac{n}{d} - 6 \right) \spmod{24}} a(n, d, m) - \sum_{m \equiv \pm \left(\mfrac{n}{d} - 18 \right) \spmod{24}} a(n, d, m) \right). \\
\sum_{m \geq 1}  \left(  \mfrac{24}{m} \right) a(n, d, m) = &\left( \mfrac{24}{n/d} \right) \left( \sum_{m \equiv \pm \mfrac{n}{d} \spmod{24}} a(n, d, m) - \sum_{m \equiv \pm \left(\mfrac{n}{d} - 12 \right) \spmod{24}} a(n, d, m) \right) \\
 &+ \left( \mfrac{24}{n/d} \right) \left( \mfrac{-4}{n} \right) \left( \mfrac{-4}{d} \right)  \left( \sum_{m \equiv \pm \left( \mfrac{n}{d} - 6 \right) \spmod{24}} a(n, d, m) - \sum_{m \equiv \pm \left(\mfrac{n}{d} - 18 \right) \spmod{24}} a(n, d, m) \right). \label{sum_m}
\end{align}
\noindent
We evaluate each of the individual sums in \eqref{sum_m}.
\begin{enumerate}
 \item For $m \equiv \pm \mfrac{n}{d} \spmod{24}$, we have
\begin{equation}\label{sum_pm_n/d}
    \sum_{m \equiv \pm \mfrac{n}{d} \spmod{24}} a(n, d, m) = \sum_{i = 0}^{\infty} a\left(2i \left( \mfrac{n}{d} - 12i \right) \right).
\end{equation}

\item For $m \equiv \pm \left(\mfrac{n}{d} - 12 \right) \spmod{24}$, we have
\begin{equation} \label{sum_pm_n/d - 12}
\displaystyle{\sum_{m \equiv \pm \left(\mfrac{n}{d} - 12 \right) \spmod{24}}} a(n, d, m) = \sum_{i = 0}^{\infty} a \left( (2i + 1) \left(\mfrac{n}{d} - 6(2i + 1)\right) \right).
\end{equation}
\item For $m \equiv \pm \left(\mfrac{n}{d} - 6\right) \spmod{24}$, we have
\begin{equation}\label{sum_pm_n/d - 6}
\displaystyle{\sum_{m \equiv \pm \left(\mfrac{n}{d} - 6 \right) \spmod{24}}} a(n, d, m) = \sum_{i = 0}^{\infty} a \left( \mfrac{(4i + 1) \left(\mfrac{n}{d} - 3(4i + 1)\right)}{2} \right). 
\end{equation}

\item For $m \equiv \pm \left(\mfrac{n}{d} - 18\right) \spmod{24}$, we have
\begin{equation}\label{sum_pm_n/d - 18}
\displaystyle{\sum_{m \equiv \pm \left(\mfrac{n}{d} - 18 \right) \spmod{24}}} a(n, d, m) = \sum_{i = 0}^{\infty} a \left( \mfrac{(4i + 3) \left(\mfrac{n}{d} - 3(4i + 3)\right)}{2} \right). 
\end{equation}

\end{enumerate}
Substituting \eqref{sum_pm_n/d}, \eqref{sum_pm_n/d - 12}, \eqref{sum_pm_n/d - 6} and \eqref{sum_pm_n/d - 18} in \eqref{sum_m}, we have
\begin{align}
\sum_{m \geq 1}  \left(  \mfrac{24}{m} \right) a(n, d, m) &=  \left( \mfrac{24}{n/d} \right)   \left( \sum_{i = 0}^{\infty} a\left(\mfrac{2id \left(n - 12id \right)}{d^2} \right)  -  \sum_{i = 0}^{\infty} a \left( \mfrac{(2i + 1)d \cdot (n - 6(2i + 1)d)}{d^2} \right) \right) \\
&+ \left(\mfrac{24}{n/d} \right) \left(\mfrac{-4}{n} \right) \left(\mfrac{-4}{d} \right)  \left(  \sum_{i = 0}^{\infty} a\left( \mfrac{(4i + 1)d \cdot \left( \mfrac{n - 3(4i + 1)d }{2} \right)}{d^2} \right) - \sum_{i = 0}^{\infty} a\left( \mfrac{(4i + 3)d \cdot \left( \mfrac{n - 3(4i + 3)d }{2} \right)}{d^2} \right)  \right). 
\end{align}
Substituting $2id = r$, $(2i + 1)d = s$, $(4i + 1)d = t$, $(4i + 3)d = u$, we deduce that 
\begin{align}
\sum_{m \geq 1}  \left(  \mfrac{24}{m} \right) a(n, d, m) &=  \left( \mfrac{24}{n/d} \right)   \left( \sum_{\substack{r \geq 0 \\ r \equiv 0 \spmod{2} \\ d \mid r}} a\left(\mfrac{r \left(n - 6r \right)}{d^2} \right)  -  \sum_{\substack{s \geq 0 \\ s \equiv 1 \spmod{2} \\ d \mid s}} a\left(\mfrac{s \left(n - 6s \right)}{d^2} \right) \right) \\
&+ \left(\mfrac{24}{n/d} \right) \left(\mfrac{-4}{n} \right) \left(\mfrac{-4}{d} \right)  \left( \sum_{\substack{t \geq 0 \\ t \equiv d \spmod{4} \\ d \mid t}} a\left( \mfrac{t \left( \mfrac{n - 3t}{2} \right)}{d^2} \right) - \sum_{\substack{u \geq 0 \\ u \equiv - d \spmod{4} \\ d \mid t}} a\left( \mfrac{u \left( \mfrac{n - 3u}{2} \right)}{d^2} \right)  \right). \label{eqn_24/m}
\end{align}
Furthermore, it follows that
\begin{align}
    \left(\mfrac{-4}{d} \right)  \left( \sum_{\substack{t \geq 0 \\ t \equiv d \spmod{4} \\ d \mid t}} a\left( \mfrac{t \left( \mfrac{n - 3t}{2} \right)}{d^2} \right) - \sum_{\substack{u \geq 0 \\ u \equiv - d \spmod{4} \\ d \mid u}} a\left( \mfrac{u \left( \mfrac{n - 3u}{2} \right)}{d^2} \right)  \right) = \sum_{\substack{t \geq 0 \\ d \mid t \\ t \equiv 1 \spmod{4}}} a\left( \mfrac{t \left( \mfrac{n - 3t}{2} \right)}{d^2} \right) - \sum_{\substack{u \geq 0 \\ d \mid u \\ u \equiv 3 \spmod{4}}} a\left( \mfrac{u \left( \mfrac{n - 3u}{2} \right)}{d^2} \right). \label{simplify_(-4/d)}
\end{align}
Substituting \eqref{simplify_(-4/d)} in \eqref{eqn_24/m}, we obtain
\begin{align}
    \sum_{m \geq 1}  \left(  \mfrac{24}{m} \right) a(n, d, m) &=  \left(\mfrac{24}{n/d} \right) \left( \sum_{\substack{r \geq 0 \\ r \equiv 0 \spmod{2} \\ d \mid r}} a\left(\mfrac{r \left(n - 6r \right)}{d^2} \right)  -  \sum_{\substack{s \geq 0 \\ s \equiv 1 \spmod{2} \\ d \mid s}} a\left(\mfrac{s \left(n - 6s \right)}{d^2} \right) \right) \\
&+ \left(\mfrac{24}{n/d} \right) \left(\mfrac{-4}{n} \right) \left(\sum_{\substack{t \geq 0 \\ d \mid t \\ t \equiv 1 \spmod{4}}} a\left( \mfrac{t \left( \mfrac{n - 3t}{2} \right)}{d^2} \right) - \sum_{\substack{u \geq 0 \\ d \mid u \\ u \equiv 3 \spmod{4}}} a\left( \mfrac{u \left( \mfrac{n - 3u}{2} \right)}{d^2} \right) \right).
\end{align}
as claimed.

\noindent
Using Lemma \ref{simplify_sum} in \eqref{exp_A1}, we obtain
\begin{align}
 B_{1}(n) &= \left( \mfrac{8}{n}  \right) \sum_{ d \mid n} \chi(d) d^{k - 1} \left( \sum_{\substack{r \geq 0 \\ r \equiv 0 \spmod{2} \\ d \mid r}} a\left(\mfrac{r \left(n - 6r \right)}{d^2} \right)  -  \sum_{\substack{s \geq 0 \\ s \equiv 1 \spmod{2} \\ d \mid s}} a\left(\mfrac{s \left(n - 6s \right)}{d^2} \right) \right) \\
 &+ \left(\mfrac{8}{n} \right) \left(\mfrac{-4}{n} \right) \sum_{ d \mid n} \chi(d) d^{k - 1} \left(\sum_{\substack{t \geq 0 \\ d \mid t \\ t \equiv 1 \spmod{4}}} a\left( \mfrac{t \left( \mfrac{n - 3t}{2} \right)}{d^2} \right) - \sum_{\substack{u \geq 0 \\ d \mid u \\ u \equiv 3 \spmod{4}}} a\left( \mfrac{u \left( \mfrac{n - 3u}{2} \right)}{d^2} \right) \right). \\
 &= \left(\mfrac{8}{n} \right) \left( \sum_{\substack{r \geq 0 \\ r \equiv 0 \spmod{2}}} \sum_{d \mid (r, n)}  \chi(d) d^{k - 1}  a\left( \mfrac{r(n - 6r)}{d^2} \right) - \sum_{\substack{s \geq 0 \\ s \equiv 1 \spmod{2}}} \sum_{d \mid (s, n)} \chi(d) d^{k - 1} a\left( \mfrac{s(n - 6s)}{d^2} \right) \right) \\
 &+ \left(\mfrac{8}{n} \right) \left(\mfrac{-4}{n} \right) \left( \sum_{\substack{t \geq 0 \\ t \equiv 1 \spmod{4}}} \sum_{d \mid (t, n)} \chi(d) d^{k - 1} a\left( \mfrac{t \left( \mfrac{n - 3t}{2} \right)}{d^2} \right) - \sum_{\substack{u \geq 0 \\ u \equiv 3 \spmod{4}}} \sum_{d \mid (u, n)} \chi(d) d^{k - 1} a\left( \mfrac{u \left( \mfrac{n - 3u}{2} \right)}{d^2} \right) \right). \\
  &= \left(\mfrac{8}{n} \right) \left( \sum_{\substack{r \geq 0 \\ r \equiv 0 \spmod{2}}} \sum_{d \mid (r, n - 6r)}  \chi(d) d^{k - 1}  a\left( \mfrac{r(n - 6r)}{d^2} \right) - \sum_{\substack{s \geq 0 \\ s \equiv 1 \spmod{2}}} \sum_{d \mid (s, n - 6s)} \chi(d) d^{k - 1} a\left( \mfrac{s(n - 6s)}{d^2} \right) \right) \\
 &+ \left(\mfrac{8}{n} \right) \left(\mfrac{-4}{n} \right) \left( \sum_{\substack{t \geq 0 \\ t \equiv 1 \spmod{4}}} \sum_{d |\left( t, \mfrac{n - 3t}{2}  \right)} \chi(d) d^{k - 1} a\left( \mfrac{t \left( \mfrac{n - 3t}{2} \right)}{d^2} \right) - \sum_{\substack{u \geq 0 \\ u \equiv 3 \spmod{4}}} \sum_{d | \left(u, \mfrac{n - 3u}{2}  \right)} \chi(d) d^{k - 1} a\left( \mfrac{u \left( \mfrac{n - 3u}{2} \right)}{d^2} \right) \right).
\end{align}
Further, using Proposition \ref{multiplicativity}, we obtain
\begin{align}
    B_{1}(n) &=  \left(\mfrac{8}{n} \right) \left(  \sum_{\substack{r \geq 0 \\ r \equiv 0 \spmod{2}}} a(r) a(n - 6r) - \sum_{\substack{s \geq 0 \\ s \equiv 1 \spmod{2}}} a(s) a(n - 6s) \right) \\
   &+ \left(\mfrac{8}{n} \right) \left(\mfrac{-4}{n} \right) \left( \sum_{\substack{ t \geq 0 \\ t \equiv 1 \spmod{4}}} a(t) a\left(\mfrac{n - 3t}{2} \right) - \sum_{\substack{u \geq 0 \\ u \equiv 3 \spmod{4}}} a(u) a\left(\mfrac{n - 3u}{2} \right) \right). \label{eqn_B1}
\end{align}
We define 
\begin{equation} \label{defn_g(a,b)}
g_{a,b}(z) = \left\{g(z) \in M_{k}(N, \chi) : g(z) =  \sum_{n \equiv a \spmod{b}} a(n) q^{n}\right\}.
\end{equation}
With $\tilde{g} = g \mid U_{2} \mid V_{2}$ and using \eqref{defn_g(a,b)}, we deduce that
\begin{align}
 g_{0,2} = \tilde{g},  g_{1,2} = g - \tilde{g},  g_{1,4} = \mfrac{1}{2}\left( g - \tilde{g} + \left(g \otimes \left(\mfrac{-4}{\cdot}  \right) \right) \right),  g_{3,4} = \mfrac{1}{2}\left( g - \tilde{g} - \left(g \otimes \left(\mfrac{-4}{\cdot}  \right) \right) \right) \label{simplify_g_(a,b)}
\end{align}
So, we conclude with
\begin{align}
 \mfrac{\eta(2z)^3}{\eta(z) \eta(4z)} g(z)    &= \sum_{n \geq 1} B_{1}(n) q^{n} \\
    &= \left( \mfrac{8}{n} \right) \left( \left( g(z) \cdot (g_{0,2} \mid V_6)(z) \right) -  \left( g(z) \cdot (g_{1,2} \mid V_6)(z) \right) \right) \\
    &+ \left( \mfrac{-8}{n} \right) \left(\left((g \mid V_2)(z)  \cdot (g_{1,4} \mid V_3)(z)\right) - ((g \mid V_2)(z)  \cdot (g_{3,4} \mid V_3)(z)) \right)  \\ 
    &= \left(g(z) \cdot \left[ \left( \left( g_{0,2} - g_{1,2}  \right) \mid V_{6} \right)(z) \right] \right) \otimes \left( \mfrac{8}{n} \right)\\
    &+ \left(\left( g \mid V_2  \right)(z) \cdot \left[ \left( \left( g_{1,4} - g_{3,4}  \right) \mid V_{3} \right)(z) \right] \right) \otimes \left( \mfrac{-8}{n} \right). \label{simplify_S_1}
\end{align}
Substituting \eqref{simplify_g_(a,b)} in \eqref{simplify_S_1}, using \eqref{Comm_V_twist} and \eqref{rankin_01}, we obtain the desired result.
\end{proof}

\subsection{Proof of Theorem \ref{first_image}}
\begin{proof}
We prove part (1) of Theorem \ref{first_image}. The proof of part (2) follows a similar approach using Lemma \ref{theta_quasi}. 

\medskip

Let $1 \leq r \leq 23$ with $(r, 6) = 1$. Let $\lambda \geq 2$ and $N \geq 1$ be integers. Let $t \geq 1$ be a square-free integer and let $\chi \in \hat{G}_{N}$. Let $g(z) \in M_{k}(N, \chi)$ be a normalized Hecke eigenform for all $T_p$. Using part (1.a) of Lemma \ref{theta_series}, it follows that
\begin{equation}
\eta(z) g(z) =  \sum_{n \geq 1} \sum_{m \geq 0} \left( \frac{12}{n} \right) \: a(m) \: q^{m + \frac{n^2}{24}} \in M_{k + 1/2}\left(N, \chi \nu_{\eta}  \right).   
\end{equation}
Plugging in $\mfrac{t}{24} = m + \frac{n^2}{24}$ implies that $t \equiv 1 \spmod{24}$. Therefore, we have
\begin{equation}
\eta(z) g(z) = \sum_{\substack{t \geq 1 \\ t \equiv 1 \spmod{24}}} b(t) \: q^{\frac{t}{24}}, \:\: \text{where} 
 \end{equation}
 \begin{equation}
\begin{aligned}
  b(t) = 
  \begin{dcases}
   \sum_{n \geq 1} \left(\frac{12}{n} \right) a \left( \frac{t - n^2}{24} \right), & \text{$t \equiv 1 \spmod{24}$}
   \\
    0,  & \text{$t \not\equiv 1 \spmod{24}$}.
\end{dcases}
\end{aligned}
\end{equation}
Using Theorem \ref{AAD1}, we have
\begin{equation}
\mathcal{S}_{1}\left(\eta(z) g(z)\right) = \mathcal{S}_{1}\left( \sum_{\substack{m \geq 1 \\ m \equiv 1 \spmod{24}}} b(m) q^{\frac{m}{24}} \right)   = \sum_{n \geq 1} B_{1}(n) q^{n},
\end{equation}
where \eqref{image_24} gives
\begin{align}
B_{1}(n) &= \sum_{d \mid n} \chi(d)  \left( \mfrac{12}{n/d} \right) d^{k - 1} b \left( \mfrac{n^2}{d^2}  \right),   \\
&= \sum_{d \mid n} \chi(d)  \left( \mfrac{12}{n/d} \right) d^{k - 1} \sum_{m \geq 1} \left(  \left(  \mfrac{12}{m} \right) a \left( \mfrac{\mfrac{n^2}{d^2} - m^2}{24}  \right) \right), \\
&= \sum_{d \mid n} \chi(d)  \left( \mfrac{12}{n/d} \right) d^{k - 1} \sum_{m \geq 1}  \left(  \mfrac{12}{m} \right) a(n, d, m) \label{exp_1A}, 
\end{align}
where we have written $a(n, d, m):= a \left(\mfrac{\mfrac{n^2}{d^2} - m^2}{24}\right)$ for brevity.

We note that due to $\left( \mfrac{12}{m} \right)$ factor in each term of \eqref{exp_1A}, the sum on the left is supported on integers $m$ with $(m, 6) = 1$. For such $m$, we have $m^2 \equiv 1 \spmod{24}$. 
Furthermore, we have
\begin{equation}
b \left( \mfrac{n^2}{d^2} \right)  \neq 0 \iff \mfrac{n^2}{d^2} \equiv 1 \spmod{24}.   
\end{equation}
This implies that $\left( \mfrac{n}{d}, 6  \right) = 1$. Therefore, we fix $n \geq 1$ and for all $d \mid n$ with $\left( \mfrac{n}{d}, 6  \right) = 1$, we deduce that
\begin{equation}
m \in \left\{\pm \mfrac{n}{d}, \:\: \pm \left(\mfrac{n}{d} - 6\right) \right\} \spmod{12}.  
\end{equation}

\begin{align}
 \sum_{m \geq 1}  \left(  \mfrac{12}{m} \right) a(n, d, m) = \left( \mfrac{12}{n/d} \right) \left( \sum_{m \equiv \pm \mfrac{n}{d} \spmod{12}} a(n, d, m) - \sum_{m \equiv \pm \left(\mfrac{n}{d} - 6 \right) \spmod{12}} a(n, d, m) \right). \label{sum_12_m}
 \end{align}
Substituting \eqref{sum_12_m} in \eqref{exp_1A}, we obtain
\begin{equation}
    B_{1}(n) = \sum_{d \mid n} \chi(d)  1_{6}\left(\frac{n}{d}\right) d^{k - 1} \left( \sum_{m \equiv \pm \mfrac{n}{d} \spmod{12}} a(n, d, m) - \sum_{m \equiv \pm \left(\mfrac{n}{d} - 6 \right) \spmod{12}} a(n, d, m) \right). \label{1_6}
\end{equation}
We first consider the case when $\frac{n}{d} \equiv 3 \spmod{6}$. This implies that $\frac{n}{d} \equiv 3, 9 \spmod{12}$ and in each of these cases, using \eqref{1_6}, it follows that the terms cancel out each other. Therefore, we deduce that
\begin{equation}
    B_{1}(n) = \sum_{d \mid n} \chi(d)  1_{2}\left(\frac{n}{d}\right) d^{k - 1} \left( \sum_{m \equiv \pm \mfrac{n}{d} \spmod{12}} a(n, d, m) - \sum_{m \equiv \pm \left(\mfrac{n}{d} - 6 \right) \spmod{12}} a(n, d, m) \right). \label{1_2}
\end{equation}
We claim that $\frac{n}{d} \equiv 1 \spmod{2}$ if and only if $v_{2}(n) = v_{2}(d)$. To observe this, let $n = 2^{v_2(n)} x$ where $x \equiv 1 \spmod{2}$. Let $d = 2^{v_2(d)} y$ where $y \equiv 1 \spmod{2}$ and let $d \mid n$.
\begin{equation} \label{n/d_odd}
1 \equiv \frac{n}{d} \equiv \frac{2^{v_2(n)} x}{2^{v_2(d)} y}  \equiv 2^{v_2(n) - v_2(d)} \:  \frac{x}{y} \equiv 2^{v_2(n) - v_2(d)} \equiv 1 \spmod{2} \iff v_{2}(n) = v_{2}(d).
\end{equation}
Therefore, we have $1_{2}\left(\frac{n}{d}\right) = 1$ and $\frac{n}{d} = \frac{x}{y}$ where $x \equiv y \equiv 1 \spmod{2}$. Using \eqref{1_2}, we deduce that
\begin{align}
    B_{1}(n) &= \sum_{y \mid x} \chi(2^{v_2(d)} y) (2^{v_2(d)} y)^{k - 1} \left( \sum_{i = 0}^{\infty} a\left(i \left( \mfrac{x}{y} - 6i \right) \right) -  \sum_{j = 0}^{\infty} a \left( \frac{(2j + 1) \left(\mfrac{x}{y} - 3(2j + 1) \right)}{2} \right)  \right). \\
    &= \sum_{y \mid x} \chi(2)^{v_2(n)} \chi(y) (2^{v_2(n)})^{k - 1} y^{k - 1} \left( \sum_{i = 0}^{\infty} a\left( \frac{iy (x - 6iy)}{y^2} \right) -  \sum_{j = 0}^{\infty} a\left( \frac{(2j + 1)y \left( \frac{x - 3 (2j + 1)y}{2}  \right)}{y^2} \right) \right).
\end{align}
Let $t \geq 1$ with $y \mid t$. Since $y \equiv 1 \spmod{2}$, we deduce that $t = (2j + 1)y \iff t \equiv y \spmod{2}$. Further, we let $s = iy$. Thus, we have
\begin{align}
    B_{1}(n) &= \sum_{y \mid x} \chi(2)^{v_2(n)} \chi(y) (2^{v_2(n)})^{k - 1} y^{k - 1} \left( \sum_{\substack{s = 0 \\ d \mid s}}^{\infty} a\left( \frac{s (x - 6s)}{y^2} \right) -  \sum_{\substack{t = 0 \\ d \mid t, \: t \: \text{odd}}}^{\infty} a\left( \frac{t \left( \frac{x - 3t}{2}  \right)}{y^2} \right) \right).\\
    &= \chi(2)^{v_2(n)} (2^{v_2(n)})^{k - 1} \left( \sum_{s = 0}^{\infty} \sum_{y \mid (s, x - 6s)} \chi(y) y^{k - 1} a\left( \frac{s (x - 6s)}{y^2} \right) -  \sum_{\substack{t = 0 \\ t \: \text{odd}}}^{\infty} \sum_{y \mid \left( t, \frac{x - 3t}{2} \right)} \chi(y) y^{k - 1} a\left( \frac{t \left( \frac{x - 3t}{2}  \right)}{y^2} \right) \right).
\end{align}
Further, using Proposition \ref{multiplicativity}, we obtain
\begin{align}\label{B_1}
    B_{1}(n) = \chi(2)^{v_2(n)} (2^{v_2(n)})^{k - 1} \left( \sum_{s = 0}^{\infty} a(s) a(x - 6s) -  \sum_{t = 0}^{\infty}  a(t) a \left( \frac{x - 3t}{2}  \right) \right).
\end{align}
Consider $G(z) = g(z) g(6z) - g(2z) g(3z) = \displaystyle{\sum_{n = 0}^{\infty}} c(n) q^{n}$. Therefore, we have
\begin{equation} \label{B1=c}
 B_{1}(n) = \chi(2)^{v_2(n)} (2^{v_2(n)})^{k - 1} c(x) = \chi(2)^{v_2(n)} (2^{v_2(n)})^{k - 1} c \left( \frac{n}{2^{v_{2}(n)}} \right) \:\: \text{since} \:\: n = 2^{v_{2}(n)}x. 
\end{equation}
We note that for all $n \geq 0$ with $(n, 6) = 1$, we have $B_{1}(n) = c(n)$. We claim that for $(n, 6) > 1$, we have $B_{1}(n) = c(n)$. We firstly show that $G(z) \in S_{2k}^{\new 2, 3}(6N, \chi^2, \epsilon_{2, 1, \chi}, \epsilon_{3, 1, \chi})$.

 We prove that $G(z)$ has eigenvalues $\epsilon_{2,1,\chi} = -\chi(2)$ and $\epsilon_{3,1,\chi} = -\chi(3)$ for the Atkin-Lehner operator $W_{2}^{6N}$ and $W_{3}^{6N}$, respectively. We prove the result for $p = 2$ case and the proof runs parallel for $p = 3$.
Let $(N, 6) = 1$. We let $g \in M_{k}(N, \chi) \subseteq M_{k}(6N, \chi \cdot 1_6)$. Using \eqref{defn_W}, we have
\begin{equation}\label{W_simplify}
    W_{2}^{6N} = \pMatrix{2a}{b}{6Nc}{2} = \pMatrix{a}{b}{3Nc}{2} \pMatrix{2}{0}{0}{1}.
\end{equation}
We compute
\begin{equation}\label{action_W2}
    G \mid_{2k} W_{2}^{6N} = g \mid_{k} W_{2}^{6N} \cdot  g \mid V_{6} \mid_k W_{2}^{6N} - g \mid V_2 \mid W_{2}^{6N} \cdot g \mid V_3 \mid_k W_{2}^{6N}.
\end{equation}
We now evaluate each component in \eqref{action_W2}, separately.
\begin{enumerate}
    \item
    \begin{align}
          g \mid_k  W_{2}^{6N} = g \mathrel{\bigg|_k} \pMatrix{a}{b}{3Nc}{2} \pMatrix{2}{0}{0}{1} = \chi(2) \: g \mathrel{\bigg|_k} \pMatrix{2}{0}{0}{1} = 2^{k/2} \: \chi(2) \: g \mid_k \: V_2.
    \label{eq:1}
    \end{align}
    \item 
    \begin{align}
    g \mid V_2 \mid_k W_{2}^{6N} = 2^{-k/2} \: g \mathrel{\bigg|} \pmatrix{2}{0}{0}{1} \mathrel{\bigg|_k} \pmatrix{2a}{b}{6Nc}{2} &= 2^{-k/2} \: g \mathrel{\bigg|} \pmatrix{2}{0}{0}{2} \mathrel{\bigg|_k} \pmatrix{2a}{b}{3Nc}{1} \\
    &= 2^{-k/2} \: \chi(1) \: g = 2^{-k/2} \: g . \label{eq:2}  
    \end{align}
 \item 
 \begin{align}
 g \mid V_3 \mid_k W_{2}^{6N} = 3^{-k/2} \: g \mathrel{\bigg|} \pmatrix{3}{0}{0}{1} \mathrel{\bigg|_k} \pmatrix{2a}{b}{6Nc}{2} &= 3^{-k/2} \: g \mathrel{\bigg|} \pmatrix{a}{3b}{Nc}{2} \mathrel{\bigg|_k} \pmatrix{6}{0}{0}{1} \\
 &= 2^{k/2} \: \chi(2) \: g \mid_k V_6. \label{eq:3}
 \end{align}
\item 
\begin{align}
    g \mid V_6 \mid W_2^{6N} = g \mid V_3 \mid V_2 \mid W_{2}^{6N} = 2^{-k/2} \:  \chi(1) \: g \mid_k V_3 = 2^{-k/2} \: g \mid_k V_3. \label{eq:4}
\end{align}
\end{enumerate}
\noindent
Substituting \eqref{eq:1}, \eqref{eq:2}, \eqref{eq:3} and \eqref{eq:4} in \eqref{action_W2}, we conclude that
\begin{equation}\label{W2_eigenvalue}
    G \mid_{2k} W_{2}^{6N} = - \chi(2)\: G = \epsilon_{2, 1, \chi}\: G.
\end{equation}

Next, we want to show that $G$ is new at primes $p \in \{2, 3\}$. Using Lemma \ref{new_criterion}, it suffices to show that $Tr_{6N/p}^{6N}(G) = 0 = Tr_{6N/p}^{6N}(G \mid_k H_{6N})$. We prove the result for $p = 2$ case and the proof follows a similar approach for $p = 3$. To begin with, we show that $Tr_{3N}^{6N}(G) = 0.$

\medskip

\noindent
Using \eqref{defn_Hecke} and $g \mid_k T_p = a(p) g$ for all primes $p$, we deduce that
\begin{equation}\label{action_Up}
 g \mid_k U_p =  a(p) g - \chi(p) p^{k - 1} g \mid_k V_p.   
\end{equation}
We compute
\begin{align}
    G \mid U_2 = g \mid_k U_2 \cdot g \mid_k V_3 - g \cdot g \mid_k V_3 \mid U_2.
\end{align}
Using \eqref{action_Up} and the commutativity of $U_2$ and $V_3$, we conclude that
\begin{equation}\label{U2_simplify}
    G \mid U_2 = \chi(2) \: 2^{k - 1} G.
\end{equation}
Using \eqref{defn_trace1}, we have
\begin{align}
Tr_{3N}^{6N}(G) &= G + \overline{\chi^2}(2) \: 2^{1 - k} \: G \mid W_{2}^{6N} \mid_k U_2,   \\
&= G + \overline{\chi^2}(2) \: 2^{1 - k} (-\chi(2)  G) \mid_k U_2, \quad \text{using} \:\: \eqref{W2_eigenvalue} \\
&= G - \overline{\chi^2}(2) \: 2^{1 - k} \: \chi(2) (\chi(2) \: 2^{k - 1} G), \quad \text{using} \:\: \eqref{U2_simplify}\\
&= 0. \\
\end{align}
Lastly, we show that $Tr_{6N/p}^{6N}(G \mid_k H_{6N}) = 0$ for $p \in \{2,3\}$. We present the proof for $p = 2$ case. The proof for $p = 3$ follows a similar approach. 

\medskip

We present the proof for any normalized Hecke eigenform $g \in S_{k}(N, \chi)$ for all $T_p$. The proof for $g \in E_k(N, \chi)$ follows a similar approach using Theorem \ref{eisen_decomp}. 

\medskip

Using Proposition \ref{decompose_newform}, there exists $Cond(\chi) \mid M \mid N$, a unique newform $h \in S_{k}(M, \chi)$, and $\alpha_d \in \C$ such that for all $d \mid \mfrac{N}{M}$, we have
\begin{equation}\label{exp_g}
g = \sum_{d \mid \mfrac{N}{M}} \alpha_d \: h \mid_k V_d.  
\end{equation}
We claim that 
\begin{equation}\label{g_Hn}
g \mid_k H_{N} = \beta \left(\mfrac{N}{M}\right)^{k/2} \sum_{d \mid \mfrac{N}{M}} \alpha_d \: d^{-k} \: h \mathrel{\bigg|_k} V_{\mfrac{N}{Md}}.
\end{equation}
We compute 
\begin{align}
    g \mid_k H_N &= \sum_{d \mid \mfrac{N}{M}} \alpha_d \: h \mid V_d \mid_k H_N, \:\: \text{using} \:\: \eqref{exp_g} \\
    &= \sum_{d \mid \mfrac{N}{M}} \alpha_d \: d^{-k/2} \: h \mid_k \pMatrix{d}{0}{0}{1} \pMatrix{0}{-1}{N}{0}.    
\end{align}
We observe that
\begin{equation}\label{matrix_simp}
    \pMatrix{d}{0}{0}{1} \pMatrix{0}{-1}{N}{0} = \pMatrix{0}{-d}{N}{0} =  \pMatrix{d}{0}{0}{d} \pMatrix{0}{-1}{M}{0} \pMatrix{\mfrac{N}{Md}}{0}{0}{1}.
    \end{equation}
Therefore, using \eqref{matrix_simp}, we deduce that
\begin{equation}
    g \mid_k H_N = \sum_{d \mid \mfrac{N}{M}} \alpha_d \: d^{-k/2} \: h \mid H_M \mid_k \pMatrix{\mfrac{N}{Md}}{0}{0}{1} = \beta \sum_{d \mid \mfrac{N}{M}} \alpha_d \: d^{-k} \left(\mfrac{N}{M}\right)^{k/2} h \mathrel{\bigg|_k} \: V_{\mfrac{N}{Md}},
\end{equation}
as claimed. The last equality follows from the fact there exists $\beta \in \C$ with $\abs{\beta} = 1$ such that $h \mid_k H_M = \beta h$.

Next, we claim that for $G = g \cdot g \mid_k V_6 - g \mid_k V_2 \cdot g \mid_k V_3 \in S_{2k}(6N, \chi^2)$, we have
\begin{align} 
G \mid_{2k} H_{6N} &= \beta^2 \left(\mfrac{N}{M}\right)^{k} \left( \sum_{d \mid \mfrac{N}{M}} \alpha_d \: d^{-k} \: h \mathrel{\bigg|_k} V_{\mfrac{6N}{Md}} \right) \left( \sum_{t \mid \mfrac{N}{M}} \alpha_t \: t^{-k} \: h \mathrel{\bigg|_k} V_{\mfrac{N}{Mt}} \right)  \\
&- \beta^2 \left(\mfrac{N}{M}\right)^{k} \left( \sum_{s \mid \mfrac{N}{M}} \alpha_s \: s^{-k} \: h \mathrel{\bigg|_k} V_{\mfrac{2N}{Ms}} \right) \left( \sum_{r \mid \mfrac{N}{M}} \alpha_r \: r^{-k} \: h \mathrel{\bigg|_k} V_{\mfrac{3N}{Mr}} \right).\label{action_H_6N}
\end{align}
We compute
\begin{align}
 G \mid_{2k} H_{6N} &= g \mid_k H_{6N} \cdot g \mid V_6 \mid_{k} H_{6N} - g \mid V_{2} \mid_k H_{6N} \cdot g \mid V_3 \mid_k H_{6N}, \\
 &= (6^{k/2} \: g \mid H_N \mid_k V_6)(6^{-k/2} \: g \mid_k H_N) - (6^{k/2} \: 2^{-k} g \mid H_N \mid_k V_2)(6^{k/2} \: 3^{-k} g \mid H_N \mid V_3),
\: \text{using} \: \eqref{comm_VH}, \\
 &= (g \mid H_N \mid_k V_6) (g \mid_k H_N) - (g \mid H_N \mid_k V_2) (g \mid H_N \mid_k V_3).
\end{align}
Using \eqref{g_Hn}, we obtain the desired result. 
We now require the following commutativity result. Let $c \mid 6$. Using \eqref{eq:1} and \eqref{eq:2}, we have
\begin{equation}
\begin{aligned}\label{comm_Vc_W2}
  h \mathrel{\bigg|_k} V_{\mfrac{cN}{Md}} \mathrel{\bigg|_k} W_{2}^{6N} = 
  \begin{dcases}
    2^{k/2} \: \chi(2) \: h \mathrel{\bigg|_k} V_{\mfrac{2cN}{Md}},
  & \text{$c \in \{ 1, 3\}$}, \\
  2^{-k/2} \: h \mathrel{\bigg|_k} V_{\mfrac{cN}{2Md}},
  & \text{$c \in \{2, 6 \}$}. \\
\end{dcases}
\end{aligned}
\end{equation}
Taking the action of $W_{2}^{6N}$ on both sides of \eqref{action_H_6N} and then substituting \eqref{comm_Vc_W2} in \eqref{action_H_6N}, we deduce that
\begin{equation}\label{comm_H6_W2}
 G \mid H_{6N} \mid_{2k} W_{2}^{6N} = - \chi(2) \: G \mid_{2k} H_{6N}.    
\end{equation}
\noindent
Since $h \in S_{k}^{new}(M, \chi)$ with $M$ odd, we have
\begin{equation}
 h \mid T_2 = a_{2}(h) h \:\: \text{and} \:\: h \mid_k U_2 = a_{2}(h) h - \chi(2) 2^{k - 1} h \mid_k V_2.    
\end{equation}
Thus, it follows that
\begin{align}
 G \mid H_{6N} \mid_{2k} W_{2}^{6N} \mid U_2 &= (- \chi(2) \: G \mid_{2k} H_{6N}) \mid U_2,  \:\: \text{using} \:\: \eqref{comm_H6_W2}\\
 &= - \chi(2) \: (\chi(2) \: 2^{k - 1} \: G \mid_{2k} H_{6N}), \\
 &= - \chi^{2}(2) \: 2^{k - 1} \: G \mid_{2k} H_{6N} \label{simplify}.
\end{align}
By definition of trace operator as in \eqref{defn_trace1}, we have
\begin{align}
Tr_{3N}^{6N}(G \mid_{2k} H_{6N}) &=  G \mid_{2k} H_{6N} + \overline{\chi^2}(2) \: 2^{1 - k} \: G \mid H_{6N} \mid W_{2}^{6N} \mid_{2k} U_2,  \\
&= G \mid_{2k} H_{6N} + \overline{\chi^2}(2) \: 2^{1 - k} (- \chi^{2}(2) \: 2^{k - 1} \: G \mid_{2k} H_{6N}), \:\: \text{using} \:\: \eqref{simplify}. \\
&= 0. 
\end{align}

Hence, we conclude that $G(z) = g(z) g(6z) - g(2z) g(3z) \in S_{2k}^{\new 2, 3}(6N, \chi^{2}, \epsilon_{2, 1, \chi},  \epsilon_{3, 1, \chi})$. 
\medskip
Lastly, we show that $B_{1}(n) = c(n)$ for all $n \geq 0$ with $(n, 6) > 1$. 
We require the following lemma.
\begin{lemma} \label{2-3-new}
Let $H(z) = \displaystyle{\sum_{n \geq 0}} r(n) q^{n} \in M_{2k}^{\new 2,3}(6N, \chi^2, - \chi(2), - \chi(3))$. Then for all $u, v \geq 0$ and for all $m, n \geq 0$, we have
\begin{enumerate}[label=(\alph*)]
    \item $r(2^u m) = \chi(2)^u \: 2^{u(k - 1)} \: r(m)$ and
     \item $r(3^v n) = \chi(3)^v \: 3^{v (k - 1)} \: r(n)$. 
\end{enumerate}
\end{lemma}
\begin{proof}
    Let $p \in \{2, 3\}$. Since $H$ is $p$-new, using Lemma \ref{new_criterion}, it follows that $Tr_{6N/p}^{6N}(H) = 0$. Using \eqref{defn_trace1}, we compute
    \begin{align}
        Tr_{6N/p}^{6N}(H) = H + \overline{\chi^{2}}(p) \: p^{1 - k} \: H \mid W_{p}^{6N} \mid U_p &= 0, \\
        H + \overline{\chi^{2}}(p) \: p^{1 - k} (-\chi(p) H) \mid U_p &= 0, \\
        H \mid U_p = \chi(p) \: p^{k - 1} \: H. 
    \end{align}
This implies that for all  $t \geq 0$, we have $H \mid U_p^{t} = \chi(p)^t \: p^{t(k - 1)} H$. Hence, we conclude that for all $t \geq 0$ and for all $s \geq 0$, we have $r(p^t s) = \chi(p)^t p^{t(k - 1)} r(s)$.
\end{proof}
Let $n = 2^u 3^v x$ with $u, v \geq 0$ and $(x, 6) = 1$. We do a case-by-case analysis as follows:
\begin{enumerate}
    \item Let $u = v = 0$. This implies that $n = x$ with $(x, 6) = 1$. By earlier work, we have $B_{1}(n) = c(n)$.
    \item Let $u \geq 1$ and $v = 0$. This implies that $n = 2^u x$. We have
    \begin{align}
    B_{1}(n) &= B_{1}(2^u x) = \chi(2)^u \: 2^{u(k - 1)} \: B_{1}(x), \: \text{using Lemma}  \: \ref{2-3-new} \\
    &=  \chi(2)^u \: 2^{u(k - 1)} \: c(x), \: \text{using part (1)} \\
    &= c(2^u x), \: \text{using Lemma}  \: \ref{2-3-new} \\
    &= c(n).
    \end{align}
    \item Let $u = 0$ and $v \geq 1$. This implies that $n = 3^v x$. A similar computation as in part (2) shows that $B_{1}(n) = c(n)$. 
    \item Let $u \geq 1$ and $v \geq 1$. This implies that $n = 2^u 3^v x$. We have
    \begin{align}
    B_{1}(n) &= B_{1}(2^u 3^v x) =  \chi(2)^u \: 2^{u(k - 1)} \: \chi(3)^v \: 3^{v(k - 1)}  \: B_1(x), \: \text{using Lemma}  \: \ref{2-3-new}  \\
    &=  \chi(2)^u \: 2^{u(k - 1)} \: \chi(3)^v \: 3^{v(k - 1)}  \: c(x), \: \text{using part (1)} \\
    &= c(2^u 3^v x), \: \text{using Lemma}  \: \ref{2-3-new} \\
    &= c(n).
    \end{align}
\end{enumerate}
Thus, we conclude that for all $u, v \geq 0$, we have $B_{1}(n) = c(n)$ for all $n \geq 0$. This completes the proof that $S_1(\eta(z) g(z)) = G(z) \in S_{2k}^{\new 2,3}(6N, \chi^2, \epsilon_{2, 1, \chi}, \epsilon_{3, 1, \chi})$.
\end{proof}
\subsection{Proof of Theorem \ref{image_(r,6) = 1}}
Let $1 \leq r \leq 23$ with $(r,6) = 1$ and let $s \geq 0$ be an even integer. Let $f(z) \in M_s(1)$ and let $k = \mfrac{r - 1}{2} + s$. We recall that $M_{k}\left(r, \left( \frac{\cdot}{r} \right)\right)$ has a basis $\{g_{1}, \ldots, g_{d}\}$ of normalized Hecke eigenforms for all $T_{p}$, so there exists $\alpha_{i}$ for $1 \leq i \leq d$ such that
    \begin{equation} \label{eq:alpha_i}
    \left( \frac{\eta(z)^{r}}{\eta(rz)} f(z) \right) \mathrel{\bigg|_k} U_{r} = \sum_{i = 1}^{d} \alpha_i g_{i}(z).
    \end{equation}
For all $1 \leq i \leq d$, we define $G_i(z) = g_i(z) g_i(6z) - g_i(2z) g_i(3z) \in S_{2k}(6r)$. We note that part (1) of Theorem \ref{first_image} implies that $G_i(z)$ is new at primes $p =2, 3$ with Atkin-Lehner eigenvalues $-\left( \mfrac{8}{r} \right)$ and $-\left( \mfrac{12}{r} \right)$, respectively. 

\medskip
\noindent
We show that $\mathcal{S}_{r}(\eta^r f) = \displaystyle{\sum_{i = 1}^{d}} \alpha_{i} G_{i}$. \\
We compute
\begin{align}
\mathcal{S}_r(\eta^{r}(z) f(z)) &= \mathcal{S}_{1}((\eta^{r}(z) f(z)) \mid_k U_r),  \quad \text{using part} \:\:\ref{eq:Comm_S1:d} \:\: \text{of Lemma} \:\: \ref{Comm_(r,6)=1} \\
& = \mathcal{S}_{1}\left(\eta(z) \left( \left( \frac{\eta^{r}(z)}{\eta(rz)} f(z) \right) \mathrel{\bigg|_k} U_{r} \right) \right), \\
& = \mathcal{S}_{1} \left( \eta(z) \left( \sum_{i = 1}^{d} \alpha_{i} g_{i}(z)    \right) \right), \quad \text{using} \:\: \eqref{eq:alpha_i} \\
&= \sum_{i = 1}^{d} \alpha_i \mathcal{S}_{1} \left( \eta(z) g_{i}(z) \right), \\
 & = \sum_{i = 1}^{d} \alpha_{i} G_{i}, \quad \text{using part} \: (1) \: \text{of Theorem}\:\: \ref{first_image}.
\end{align}
Next, we prove that $\mathcal{S}_{r}(\eta^r f) = \sum \alpha_{i} G_{i}$ has level 6.
Combining part \ref{eq:Comm_S1:b} and part \ref{eq:Comm_S1:d} of Lemma \ref{Comm_(r,6)=1}, we deduce that
\begin{equation} \label{relation_S1_Sr}
\Sh_1 \left( \eta(z) \left[ \left(\mfrac{\eta^r(z)}{\eta(rz)}  f(z) \right) \mathrel{\bigg|_k} U_r \right]  \right)   = \Sh_{r}(\eta^r(z) f(z)) = \left( \mfrac{12}{r} \right) \: \Sh_1 \left( \eta(z) \left( \mfrac{\eta^r(rz)}{\eta(z)} f(rz) \right)  \right) \mathrel{\bigg|_{2k}} U_r.
\end{equation}
Further, we compute 
\begin{align}
   \left( \mfrac{\eta^r(z)}{\eta(rz)} f(z) \right) \mathrel{\bigg|_k} H_r &= \mfrac{r^{(k + 1)/2}}{i^{(r-1)/2}} \: \mfrac{\eta^r(rz)}{\eta(z)} f(rz), \label{eq:eta_Hr1} \\
     \left( \mfrac{\eta^r(rz)}{\eta(z)} f(rz) \right) \mathrel{\bigg|_k} H_r &= \mfrac{1}{r^{(k + 1)/2} \cdot i^{(r-1)/2}} \: \mfrac{\eta^r(z)}{\eta(rz)} f(z), \label{eq:eta_Hr2} \\
     E_{k}^{1_1, \left( \mfrac{\cdot}{r} \right)} \mathrel{\bigg|_k} H_r &= \mfrac{r^{k/2}}{\tau\left( \left( \mfrac{\cdot}{r} \right)\right)} \: E_{k}^{\left( \mfrac{\cdot}{r} \right), 1_1}, \label{eq:Ek_1_Hr} \\
    E_{k}^{\left( \mfrac{\cdot}{r} \right), 1_1} \mathrel{\bigg|_k} H_r &= \mfrac{\left(  \mfrac{-1}{r} \right) \: \tau\left( \left( \mfrac{\cdot}{r} \right)\right)}{r^{k/2}} \: E_{k}^{1_1, \left( \mfrac{\cdot}{r} \right)}, \label{eq:Ek_2_Hr}
\end{align}
where $\tau\left( \left( \mfrac{\cdot}{r} \right)\right)$ denotes the Gauss sum.
We recall that $M_{k}\left( r, \left( \mfrac{\cdot}{r} \right)  \right)$ has a basis of normalized Hecke eigenforms for all $T_p$ given by \\
$\left\{g_1 = E_{k}^{1_1, \left( \mfrac{\cdot}{r} \right)}, \: g_2 = E_{k}^{\left( \mfrac{\cdot}{r} \right), 1_1}, \: g_3, \: \ldots \:, \: g_d  \right\}$ where $g_{i} \in S_{k}^{\new}\left( r, \left( \mfrac{\cdot}{r} \right)  \right)$ for $3 \leq i \leq d$. 
Let $g_i(z) = \displaystyle{\sum_{n}} a_{i}(n) q^{n}$ for all $1 \leq i \leq d$. This implies that for all $3 \leq i \leq d$, we have
\begin{equation} \label{prod_coeff}
a_{i}(r) \: \overline{a_{i}(r)} = r^{k-1},  
\end{equation} 
and using \cite{weisinger} and \cite{winnie-li}, we deduce that 
\begin{equation} \label{gi_Hr}
g_{i} \mid_k H_r = \lambda_{i} \: \overline{g_i}, \:\: \text{where}
\end{equation}
\begin{enumerate}
    \item \begin{alignat}{2}
& \begin{aligned}
  &\overline{g_i} = 
&  \begin{cases}
   g_2, & i = 1, \\
   g_1, & i = 2, \\
   g_i \mid K = \overline{g_i(- \overline{z})}, & 3 \leq i \leq d. \\
  \end{cases} \\
  \end{aligned}
\quad \text{and}   \quad
\begin{aligned} \label{eval_eigen}
  &\lambda_i = 
 & \begin{cases}
  \mfrac{r^{k/2}}{\tau \left(\left( \mfrac{\cdot}{r}  \right)\right)}, & i = 1, \\ \\
  \mfrac{\left( \mfrac{-4}{r} \right) \: \tau \left(\left( \mfrac{\cdot}{r}  \right)\right)}{r^{k/2}}, & i = 2, \\ \\
   \mfrac{r^{1 - k/2} \:\: \overline{a_i(r)}}{\tau \left(\left( \mfrac{\cdot}{r}  \right)\right)} , & 3 \leq i \leq d. \\
  \end{cases} \\
  \end{aligned}
\end{alignat} 
\item    \begin{equation} \label{action_gi_Hr}
    \overline{g_i} \mid_k H_r = \left( \mfrac{-4}{r} \right) \lambda_i^{-1} g_i.
    \end{equation}
\end{enumerate}

We note that $\left\{ \overline{g_1} = E_{k}^{\left( \mfrac{\cdot}{r} \right), 1_1}, \: \overline{g_2} = E_{k}^{1_1, \left( \mfrac{\cdot}{r} \right)}, \: \overline{g_3}, \: \ldots \:, \: \overline{g_d}  \right\}$ is also a basis of normalized Hecke eigenforms of the space $M_{k}\left( r, \left( \mfrac{\cdot}{r} \right)  \right)$ for all $T_p$. We suppose that
\begin{align} 
    \mfrac{\eta^r(rz)}{\eta(z)} f(rz) = \displaystyle{\sum_{j = 1}^{d} \beta_j \: \overline{g_j}}. \label{eq:beta_j} 
\end{align}
We claim that for all $1 \leq i \leq d$, we have 
\begin{equation}
\alpha_i = r^{(k + 1)/2} \cdot i^{(r-1)/2} \: \beta_i \: \left( \mfrac{-4}{r} \right) \:\lambda_i^{-1} \: a_i(r). \label{rel_alpha_beta}
\end{equation}
On the one hand,
\begin{align}
\left( \mfrac{\eta^r(rz)}{\eta(z)} f(rz) \right) \mathrel{\bigg|} H_r \mathrel{\bigg|_k} U_r &= \mfrac{1}{r^{(k + 1)/2} \cdot i^{(r-1)/2}} \: \left(\mfrac{\eta^r(z)}{\eta(rz)} f(z) \right) \mathrel{\bigg|_k} U_r, \: \text{using} \: \eqref{eq:eta_Hr2}\\
&= \mfrac{1}{r^{(k + 1)/2} \cdot i^{(r-1)/2}} \: \left(\displaystyle{\sum_{i = 1}^{d} \alpha_i \: g_i} \right), \: \text{using} \: \eqref{eq:alpha_i} \label{alphai_g}.
\end{align}
On the other hand,
\begin{align}
\left( \mfrac{\eta^r(rz)}{\eta(z)} f(rz) \right) \mathrel{\bigg|} H_r \mathrel{\bigg|_k} U_r &= \left( \displaystyle{\sum_{j = 1}^{d} \beta_j \: \overline{g_j}} \right) \mathrel{\bigg|} H_r \mathrel{\bigg|_k} U_r, \: \text{using} \:\eqref{eq:beta_j} \\
&= \displaystyle{\sum_{j = 1}^{d}}  \beta_j \left( \left( \mfrac{-4}{r} \right) \lambda_j^{-1} g_j \right) \mathrel{\bigg|_k} U_r, \: \text{using} \: \eqref{action_gi_Hr}\\
&= \displaystyle{\sum_{j = 1}^{d}}  \beta_j \left( \left( \mfrac{-4}{r} \right) \lambda_j^{-1} \right) \: (a_j(r) \: g_j).  \label{betaj_g}
\end{align}
Comparing \eqref{alphai_g} and \eqref{betaj_g}, we obtain the desired result in \eqref{rel_alpha_beta} as claimed.
We define the following for all $1 \leq i \leq d$,
\begin{equation} \label{defn_Gi}
    G_{i}(z) = g_{i}(z) \: g_{i}(6z) - g_{i}(2z) \: g_{i}(3z) \quad , \quad \overline{G_{i}}(z) = \overline{g_{i}}(z) \: \overline{g_{i}}(6z) - \overline{g_{i}}(2z) \: \overline{g_{i}}(3z)
\end{equation}
Using \eqref{comm_V_W_6N}, \eqref{gi_Hr}, \eqref{action_gi_Hr}, and \eqref{defn_Gi},  it follows that 
\begin{equation} \label{G_Wr}
    G_{i} \mid_{2k} W_{r}^{6r} = \left( \mfrac{6}{r} \right) \: \lambda_{i}^{2} \: \overline{G_i} \quad \text{and} \quad \overline{G_{i}} \mid_{2k} W_{r}^{6r} = \left( \mfrac{6}{r} \right) \: \lambda_{i}^{-2} \: G_i.
\end{equation}
We require the following proposition.
\begin{proposition}  \label{comm_Sr_Wr}  Let $f(z) \in M_{k}(N, \chi)$.
 \begin{enumerate} [label=(\alph*)]
     \item $\Sh_r(\eta^r(z) f(z)) \mid_{2k} W_{r}^{6r} = \left( \mfrac{12}{r} \right) \: r^{k} \: \left(\displaystyle{\sum_{i = 1}^{d}} \beta_i \: \overline{G_{i}} \right)$. \label{eq:comm_Sr_Wr:a}
     \item $\left( \displaystyle{\sum_{i = 1}^{d}} \beta_i \overline{G_{i}} \right) \mathrel{\bigg|_k} U_{r} = \left( \mfrac{12}{r} \right) \: \Sh_{r}(\eta^{r}(z) f(z))$. \label{eq:comm_Sr_Wr:b}
 \end{enumerate}   
\end{proposition}
\begin{proof}
To prove part \ref{eq:comm_Sr_Wr:a} of Proposition \ref{comm_Sr_Wr}, we compute
\begin{align}
 \Sh_r(\eta^r(z) f(z)) \mid_{2k} W_{r}^{6r} &= \Sh_1 \left( \eta(z) \left[\left( \mfrac{ \eta^r(z)}{\eta(rz)} f(z) \right) \mathrel{\bigg|_{k}} U_r  \right]  \right) \mathrel{\bigg|_{2k}} W_{r}^{6r},  \quad \text{using part} \:\: \ref{eq:Comm_S1:d} \:\:\text{of Lemma} \:\: \ref{Comm_(r,6)=1} \\
 &= \Sh_1 \left( \eta(z) \cdot \sum_{i = 1}^{d} \alpha_{i} g_{i}(z) \right) \mathrel{\bigg|_{2k}} W_{r}^{6r}, \:\: \text{using} \: \eqref{eq:alpha_i} \\
 &= \left( \sum_{i = 1}^{d} \alpha_i \: G_{i}(z) \right) \mathrel{\bigg|_{2k}} W_{r}^{6r}, \quad \text{using part} \: (1) \: \text{of Theorem}\:\: \ref{first_image}.\\
 &= \sum_{i = 1}^{d} \alpha_i \left( \mfrac{6}{r} \right) \: \lambda_{i}^{2} \: \overline{G_i}(z), \:\: \text{using} \: \eqref{G_Wr} \\
 &= \sum_{i = 1}^{d} \left(  r^{(k + 1)/2} \cdot i^{(r-1)/2} \: \left( \mfrac{-6}{r} \right) \:\lambda_i \: a_i(r) \right) \:\: \beta_i \: \overline{G_i}(z), \quad \text{using} \: \eqref{rel_alpha_beta}. \label{simplify_Sr_W} \\
\end{align}
We claim that for all $1 \leq i \leq d$, we have 
\begin{equation} \label{simplify_coeff}
\left( \mfrac{-6}{r} \right) \: r^{(k + 1)/2} \cdot i^{(r-1)/2} \:\lambda_i \: a_i(r) = \left( \mfrac{12}{r} \right) \: r^{k}.
\end{equation}
To prove the claim, we first observe that for $(r, 6) = 1$, 
\begin{alignat}{2}
& \begin{aligned}
  &\tau \left( \left( \mfrac{\cdot}{r} \right)  \right)  = 
&  \begin{cases}
   r^{1/2}, & r \equiv 1 \spmod{4}, \\
   i \: r^{1/2}, & r \equiv 3 \spmod{4}. \\
  \end{cases} \\
  \end{aligned}
\quad \implies \quad 
\begin{aligned} \label{eval_gauss}
 & \begin{cases} 
  \mfrac{\tau \left( \left( \mfrac{\cdot}{r} \right)  \right)}{i^{(r - 1)/2}} = \left( \mfrac{-8}{r} \right) \: r^{1/2}, \\ \\
  \tau \left( \left( \mfrac{\cdot}{r} \right)  \right) \: i^{(r - 1)/2} = \left( \mfrac{8}{r} \right) \: r^{1/2}.  \\
  \end{cases} \\
  \end{aligned}
\end{alignat} 
We now do a case-by-case analysis as follows:
\begin{enumerate}[label=(\alph*)]
    \item Let $i = 1$. We compute
    \begin{align}
 \left( \mfrac{-6}{r} \right) \: r^{(k + 1)/2} \cdot i^{(r-1)/2} \:\lambda_i \: a_i(r) &=         \left( \mfrac{-6}{r} \right) \: r^{(k + 1)/2} \cdot i^{(r-1)/2} \mfrac{r^{k/2}}{\tau \left(\left( \mfrac{\cdot}{r}  \right)\right)}, \:\: \text{using} \:\: \eqref{eval_eigen}\\
 &= \left( \mfrac{-6}{r} \right) \: r^{k + 1/2} \: \left( \mfrac{-8}{r} \right) \: r^{-1/2}, \:\: \text{using} \:\: \eqref{eval_gauss} \\
 &= \left(  \mfrac{12}{r} \right) \: r^{k}.
    \end{align}
    \item  Let $i = 2$. We compute
    \begin{align}
    \left( \mfrac{-6}{r} \right) \: r^{(k + 1)/2} \cdot i^{(r-1)/2} \:\lambda_i \: a_i(r) &=         \left( \mfrac{-6}{r} \right) \: r^{(k + 1)/2} \cdot i^{(r-1)/2} \: \mfrac{\left( \mfrac{-4}{r} \right) \: \tau \left(\left( \mfrac{\cdot}{r}  \right)\right)}{r^{k/2}} \: r^{k-1}, \:\: \text{using} \:\: \eqref{eval_eigen}\\
    &= \left( \mfrac{-6}{r} \right) \: r^{k - 1/2} \: \left( \mfrac{-4}{r} \right) \: \left( \mfrac{8}{r} \right) \: r^{1/2}, \:\: \text{using} \:\: \eqref{eval_gauss}\\
    &= \left(  \mfrac{12}{r} \right) \: r^{k}.
 \end{align}
 \item  Let $3 \leq i \leq d$. We compute
    \begin{align}
    \left( \mfrac{-6}{r} \right) \: r^{(k + 1)/2} \cdot i^{(r-1)/2} \:\lambda_i \: a_i(r) &=         \left( \mfrac{-6}{r} \right) \: r^{(k + 1)/2} \cdot i^{(r-1)/2} \:  \mfrac{r^{1 - k/2} \:\: \overline{a_i(r)}}{\tau \left(\left( \mfrac{\cdot}{r}  \right)\right)} \: a_{i}(r), \:\: \text{using} \:\: \eqref{eval_eigen}\\
    &= \left( \mfrac{-6}{r} \right) \:  r^{(k + 1)/2} \: \left( \mfrac{-8}{r} \right) \: r^{-1/2} \: r^{1 - k/2} \: r^{k-1},  \:\: \text{using} \:\: \eqref{eval_gauss} \:\: \text{and} \:\: \eqref{prod_coeff}\\
    &= \left(  \mfrac{12}{r} \right) \: r^{k}.
 \end{align} 
\end{enumerate}
Substituting \eqref{simplify_coeff} in \eqref{simplify_Sr_W}, we obtain the desired result.
We now proceed to the proof of \ref{eq:comm_Sr_Wr:b} of Proposition \ref{comm_Sr_Wr}.  We compute
\begin{align}
 \Sh_r(\eta^r(z) f(z)) &= \left( \mfrac{12}{r}  \right) \: \Sh_1 \left( \eta(z) \left( \mfrac{ \eta^r(rz)}{\eta(z)} f(rz) \right) \right) \mathrel{\bigg|_{2k}}  U_r,  \quad \text{using} \:\: \eqref{relation_S1_Sr}  \\
 &= \left( \mfrac{12}{r}  \right) \: \Sh_1 \left( \eta(z) \cdot \sum_{j = 1}^{d} \beta_{j} \overline{g_{j}} \right) \mathrel{\bigg|_{2k}} U_r, \quad \text{using} \:\: \eqref{eq:beta_j}\\
  &= \left( \mfrac{12}{r}  \right) \left( \sum_{j = 1}^{d} \beta_{j} \overline{G_j} \right)\mathrel{\bigg|_{2k}} U_r, \quad \text{using part} \: (1) \: \text{of Theorem}\:\: \ref{first_image}. \label{LHS_Sr}
 \end{align}
Multiplying both sides of \eqref{LHS_Sr} by $\left( \mfrac{12}{r}  \right)$ gives us the desired result. 
\end{proof}
\noindent
Lastly, we claim that $\Sh_r(\eta^r(z) f(z)) \in S_{2k}(6)$. It suffices to show that 
\begin{equation}
Tr_{6}^{6r}(\Sh_{r}(\eta^r(z) f(z))) = [ \Gamma_{0}(6): \Gamma_{0}(6r)] \: \Sh_r(\eta^r(z) f(z)) = (r + 1) \: \Sh_r(\eta^r(z) f(z)).
\end{equation}
We compute
\begin{align}
Tr_{6}^{6r}(\Sh_{r}(\eta^r(z) f(z))) &=  \Sh_{r}(\eta^r(z) f(z)) + r^{1 - k} \: \Sh_{r}(\eta^r(z) f(z)) \mid_{2k} W_{r}^{6r} \mid U_{r}, \:\: \text{using} \:\: \eqref{defn_trace1}\\
&= \Sh_{r}(\eta^r(z) f(z)) + r^{1 - k} \: \left( \left( \mfrac{12}{r} \right) \: r^{k} \: \displaystyle{\sum_{i = 1}^{d}} \beta_i \: \overline{G_{i}} \right) \mid U_{r}, \: \text{using part} \:\ref{eq:comm_Sr_Wr:a} \: \text{of Proposition} \: \ref{comm_Sr_Wr} \\
&= \Sh_{r}(\eta^r(z) f(z)) + r \: \left( \mfrac{12}{r} \right) \left( \left( \mfrac{12}{r} \right) \: \Sh_{r}(\eta^{r}(z) f(z))  \right), \: \text{using part} \:\ref{eq:comm_Sr_Wr:b} \: \text{of Proposition} \: \ref{comm_Sr_Wr}\\
&= (1 + r) \: \Sh_{r}(\eta^r(z) f(z)).
\end{align}
This completes the proof of Theorem \ref{image_(r,6) = 1}.

\subsection{Proof of Theorem \ref{image_(3|r)}}
Let $r$ be an odd integer with $0 < r < 8$. and let $s \geq 0$ be an even integer. Let $f(z) \in M_s(1)$ and let $k = \mfrac{3(r - 1)}{2} + s$. We recall that $M_{k}\left(r, \left( \frac{\cdot}{r} \right)\right)$ has a basis $\{g_{1}, \ldots, g_{d}\}$ of normalized Hecke eigenforms for all $T_{p}$, so there exists $\alpha_{i}$ for $1 \leq i \leq d$ such that
    \begin{equation} \label{eq:alpha_eta3}
    \left( \frac{\eta^{3r}(z)}{\eta^{3}(rz)} f(z) \right) \mathrel{\bigg|_{k}} U_{r} = \sum_{i = 1}^{d} \alpha_i g_{i}(z).
    \end{equation}
    For all $1 \leq i \leq d$, we define $G_i(z) = g_i(z) \: \Theta(g_i(2z)) - g_i(z) \: \Theta(g_i(3z)) \in S_{2k + 2}(2r)$. We note that part (2) of Theorem \ref{first_image} implies that $G_i(z)$ is new at prime $p =2$ with Atkin-Lehner eigenvalue $-\left( \mfrac{8}{r} \right)$.
\noindent
We observe that for $r = 3$,
\begin{align}
\left( \frac{\eta^{9}(z)}{\eta^{3}(3z)} f(z) \right) \mathrel{\bigg|} U_{3} = \frac{\left( \eta^{9}(z) f(z) \right) \mid U_{3}} {\eta^{3}(z)}  \in  
 M_{r + s}\left( 3, \left( \frac{\cdot}{3} \right) \right) \implies 
\left( \eta^{9}(z) f(z) \right) \mid U_{3} \in 
 S_{\frac{9}{2} + s}\left( 3, \left( \frac{\cdot}{3} \right) \nu_{\eta}^{3} \right).
\end{align}
and for $r \in \{1,5,7\}$, we have
\begin{equation}
\left( \frac{\eta^{3r}(z)}{\eta^{3}(rz)} f(z) \right) \mathrel{\bigg|} U_{r} = \frac{\left( \eta^{3r}(z) f(z) \right) \mid U_{r}} {\eta^{3}(z)}  \in  
 M_{k}\left( 3, \left( \frac{\cdot}{3} \right) \right) \implies 
\left( \eta^{3r}(z) f(z) \right) \mid U_{r} \in 
 S_{\frac{3r}{2} + s}\left( r, \left( \frac{\cdot}{r} \right) \nu_{\eta}^{3} \right).
\end{equation}

\noindent
We show that $\mathcal{S}_{r}(\eta^{3r}(z) f(z)) = \displaystyle{\sum_{i = 1}^{d}} \alpha_{i} G_{i}(z)$. 
\noindent
We compute
\begin{align}
\mathcal{S}_r(\eta^{3r}(z) f(z)) &= \mathcal{S}_{1}((\eta ^{3r}(z) f(z)) \mid_k U_r),  \quad \text{using part} \: (b) \:\:\text{of Lemma} \:\: \ref{Comm_r/3} \\
& = \mathcal{S}_{1}\left(\eta^3(z) \left( \left( \frac{\eta^{3r}(z)}{\eta^{3}(rz)} f(z) \right) \mathrel{\bigg|_{k}} U_{r} \right) \right), \\
& = \mathcal{S}_{1} \left( \eta^{3}(z) \left( \sum_{i = 1}^{d} \alpha_{i} g_{i}(z)    \right) \right) \text{using} \:\: \eqref{eq:alpha_eta3}\\
&= \sum_{i = 1}^{d} \alpha_i \: \mathcal{S}_{1} \left( \eta^{3}(z) g_{i}(z) \right), \\
 & = \sum_{i = 1}^{d} \: \alpha_{i} G_{i}(z), \text{using part} \: (2) \: \text{of Theorem}\:\: \ref{first_image}.
\end{align}

To prove that $\mathcal{S}_{r}(\eta^{3r} f) = \sum \alpha_{i} G_{i}$ has level 2, we use a similar approach as in the proof of part (1) of Theorem \ref{image_(r,6) = 1}.

\medskip

The proof of Theorem \ref{rankin_bracket} follows from an adaptation of Xui's work \cite{huixue} and making use of Theorem \ref{first_Shimura_image}.


\newpage

\begin{thebibliography}{9}

\bibitem{AAD}
S. Ahlgren, N. Anderson and R. Dicks, \emph{The {S}himura lift and congruences for modular forms with the eta multiplier}, Adv. Math. {\bf 447} (2024), 58 pp. {doi: 10.1016/j.aim.2024.109655}

\bibitem{ABR}
S. Ahlgren, O. Beckwith and M. Raum, \emph{Scarcity of congruences for the partition function}, Amer. J. Math. {\bf 143(5)} (2023), 1509--1548. {doi: 10.1353/ajm.2023.a907704}


\bibitem{lin-lola}
B. Linowitz and L. Thompson, \emph{The {F}ourier coefficients of {E}isenstein series newforms}, Contemp. Math. {\bf 732} (1978), 169--176. {doi: 10.1090/conm/732/14788}

\bibitem{cipra}
B. A. Cipra, \emph{On the {S}himura lift, apr\`es {S}elberg}, J. Number Theory {\bf 32(1)} (1989), 58--64. {doi:10.1016/0022-314X(89)90097-8}

\bibitem{Cohen}
H. Cohen, \emph{Sums involving the values at negative integers of {$L$}-functions of quadratic characters}, Math. Ann.  {\bf 217(3)} (1975), 271--285. {doi:10.1007/BF01436180}

\bibitem{Cohen2017ModularFA}
H. Cohen and F. Str\"omberg, \emph{Modular Forms: A Classical Approach}, Graduate Studies in Mathematics {\bf 179} (2017), American Mathematical Society, Providence, RI, xii+700 pp. {doi: 10.1090/gsm/179}

\bibitem{garvan}
F.G. Garvan, \emph{Congruences for {A}ndrews' smallest parts partition function and new congruences for {D}yson's rank}, Int. J. Number Theory {\bf 6(2)} (2010), 281--309. {doi: 10.1142/S179304211000296X}

\bibitem{guo-ono}
L. Guo and K. Ono, \emph{The partition function and the arithmetic of certain modular {$L$}-functions}, Internat. Math. Res. Notices {\bf 21} (1999), 1179--1197. {doi: 10.1155/S1073792899000641}

\bibitem{HN}
D. Hansen and Y. Naqvi, \emph{Shimura lifts of half-integral weight modular forms arising
from theta functions}, Ramanujan J. {\bf 17(3)} (2008), 343--354. {doi:10.1007/s11139-007-9020-1}

\bibitem{kohnen-zagier}
W. Kohnen and D. Zagier, \emph{Values of {$L$}-series of modular forms at the center of the critical strip}, Invent. Math., {\bf 64(2)} (1981), 175--198. {doi: 10.1007/BF01389166}

\bibitem{martin}
K. Martin, \emph{Refined dimensions of cusp forms, and equidistribution and bias of signs}, J. Number Theory, {\bf 188}(2018), 1--17. {doi: 10.1016/j.jnt.2018.01.015}

\bibitem{mersmann}
G. Mersmann, \emph{Holomorphe $\eta$-Produkte und nichtverschwindende ganze Modulformen für $\Gamma_{0}(n)$}, Master's Thesis, Univ. of Bonn, 1991.


\bibitem{oliver}
R. J. L. Oliver, \emph{Eta-quotients and theta functions}, Adv. Math., {\bf 241}(2013), 1--17. {doi: 10.1016/j.aim.2013.03.019}

\bibitem{PR}
M. K. Pandey and B. Ramakrishnan, \emph{Shimura lifts of certain classes of modular forms of
half-integral weight}, Int. J. Number Theory {\bf 17(6)} (2021), 1391--1404. {doi: 10.1142/S1793042121500421}

\bibitem{weisinger}
J. Weisinger, \emph{Some results on classical Eisenstein series and modular forms over function field}, Harvard thesis. (1977)

\bibitem{webb}
J. J. Webb, \emph{Partition values and central critical values of certain modular {$L$}-functions}, Proc. Amer. Math. Soc. {\bf 138(4)} (2010), 1263--1272. {doi: 10.1090/S0002-9939-09-10188-0}

\bibitem{winnie-li}
W. C. W. Li, \emph{Newforms and functional equations}, Math. Ann. {\bf 212} (1975), 285--315. {doi: 10.1007/BF01344466}

\bibitem{yang1}
Y. Yang, \emph{Modular forms of half-integral weights on {$\rm SL(2,\Z)$}}, Nagoya Math. J. {\bf 215} (2014), 1--66. {doi: 10.1215/00277630-2684452}

\bibitem{wang}
W. Wang, \emph{Shimura lift of Rankin-Cohen brackets of eigenforms and theta series}, 2025.
url={https://arxiv.org/abs/2406.14254}

\bibitem{waldspurger1}
J.-L. Waldspurger,  \emph{Sur les coefficients de {F}ourier des formes modulaires de poids demi-entier}, J. Math. Pures Appl. (9) {\bf 60(4)} (1981), 375--484.

\bibitem{waldspurger2}
J.-L. Waldspurger,  \emph{Correspondance de {S}himura}, J. Math. Pures Appl. (9) {\bf 59(1)} (1980), 1--132.

\bibitem{huixue}
H. Xue, \emph{A {S}elberg identity for the {S}himura lift}, Bull. Lond. Math. Soc., {\bf 56(11)}(2024), 3565--3579. {doi: 10.1112/blms.13151}

\bibitem{yang2}
Y. Yang, \emph{Modular forms of half-integral weights on $\SL_{2}(\Z)$}, Nagoya Math. J., {\bf 215}(2014), 1--66.

\bibitem{zagier}
D. Zagier, \emph{Modular forms whose Fourier coefficients involve zeta-functions of quadratic fields}, Modular functions of one variable, VI (Proc. Second Internat. Conf., Univ. Bonn, Bonn, 1976), {\bf 627} Lecture Notes in Math.,1977, 105–-169.

\end{thebibliography}
\end{document}